\documentclass[12pt,english]{article}
\usepackage[T1]{fontenc}
\usepackage[utf8]{inputenc}
\usepackage{geometry}
\usepackage{amsthm}

\usepackage{amsmath,amssymb,amsfonts,amscd}
\usepackage{color,xcolor}
\usepackage{graphicx,subfigure}
\usepackage{manfnt}
\usepackage{picinpar}
\usepackage{cases}
\usepackage{extarrows}
\usepackage{epstopdf}
\usepackage{authblk}
\graphicspath{{fig/}}
\usepackage{babel}

\makeatletter
\@ifundefined{date}{}{\date{}}

\newcommand{\rom}[1]{\uppercase\expandafter{\romannumeral#1}}

\newcommand{\beq}{\begin{equation}}
\newcommand{\eeq}{\end{equation}}
\newcommand{\bal}{\begin{align}}
\newcommand{\eal}{\end{align}}
\newcommand{\baln}{\begin{align*}}
\newcommand{\ealn}{\end{align*}}

\theoremstyle{plain}\newtheorem{theorem}{Theorem}[section]
\theoremstyle{definition}\theoremstyle{plain}\newtheorem{proposition}[theorem]{Proposition}\newtheorem{corollary}{Corollary}

\newcommand{\order}{\mathcal{O}}

\makeatother

\title{Asymptotic-preserving particle-in-cell method for the magnetized Vlasov--Poisson--Fokker--Planck equation}

\author[1,2]{Anjiao Gu}
\author[3,*]{Xiaojiang Zhang}
\affil[1]{\small School of Mathematical Sciences, Shanghai Jiao Tong University, Shanghai 200240, China}
\affil[2]{\small Institute of Natural Sciences, Shanghai Jiao Tong University, Shanghai 200240, China}
\affil[3]{\small School of Mathematics and Statistics,  Central South University, Changsha 410083, China}
\affil[*]{\small Corresponding author: $\text{zhang\_xiaojiang}$@csu.edu.cn}

\date{}  

\begin{document}

\maketitle

\abstract{
In this work, we develop and rigorously analyze a new class of particle methods for the magnetized Vlasov--Poisson--Fokker--Planck system. 
The proposed approach addresses two fundamental challenges: (1) the curse of dimensionality, which we mitigate through particle methods while preserving the system's asymptotic properties, and (2) the temporal step size limitation imposed by the small Larmor radius in strong magnetic fields, which we overcome through semi-implicit discretization schemes.
We establish the theoretical foundations of our method, proving its asymptotic-preserving characteristics and uniform convergence through rigorous mathematical analysis. 
These theoretical results are complemented by extensive numerical experiments that validate the method's effectiveness in long-term simulations. Our findings demonstrate that the proposed numerical framework accurately captures key physical phenomena, particularly the magnetic confinement effects on plasma behavior, while maintaining computational efficiency.
}


\section{Introduction}
The present study investigates plasma dynamics in strongly magnetized systems, with particular emphasis on the combined effects of magnetic confinement and Coulomb collisions. 
At the kinetic level, these systems are fundamentally described by the Vlasov equation, which determines the temporal and spatial evolution of the particle distribution function $f(t,x,v)$. 
While fluid models provide computationally tractable solutions for equilibrium conditions, they often prove inadequate for capturing non-equilibrium phenomena, making kinetic modeling essential for accurate physical representation.

Numerical approaches for solving Vlasov-type equations primarily fall into two categories: Eulerian methods and particle methods.
The Eulerian approach, also referred to as the grid-based method, involves discretizing the partial differential equations on a fixed computational grid and performing time integration of the distribution function at mesh points. This class of methods has seen extensive development, with various implementations successfully applied to plasma physics problems. 
These include finite difference methods \cite{qiu2017high}, finite volume methods \cite{Filbet2001Convergence,vogman2018conservative}, finite element methods \cite{na2018relativistic,heath2012discontinuous,cheng2013study}, and spectral methods \cite{blaustein2024structure,cai2013solving,di2019filtered}, among others.
In contrast to grid-based methods, the Particle-in-Cell (PIC) technique adopts a fundamentally different approach. It utilizes 'super particles' to approximate the distribution function through weighted Klimontovich representation and tracks their trajectories through phase space. This method proves particularly advantageous for kinetic-level plasma simulations, where the high dimensionality of phase space presents significant computational challenges. The PIC method demonstrates relatively low computational costs for high-dimensional simulations \cite{Nicolas2009,LYZ2019}, making it an attractive choice for such problems.
While the use of super particles introduces computational noise, empirical studies have established that this noise decreases proportionally to $1/\sqrt{N_p}$, where $N_p$ represents the number of super particles \cite{tskhakaya2007particle}. 
Despite this limitation, PIC simulations have proven effective in reproducing realistic physical phenomena \cite{Hockney1988} and are widely recognized as an efficient framework for kinetic plasma simulations \cite{Cottet2000}.

Simulating kinetic-level plasma phenomena presents another significant challenge: the inherent multi-scale nature of these systems. 
Characteristic scales such as the Debye length, Larmor radius, and mean free path often differ by orders of magnitude from the macroscopic dimensions of the device. 
Furthermore, the relaxation time introduces an additional temporal scale, further complicating the system dynamics. 
From a computational perspective, traditional explicit numerical schemes face substantial limitations, as they require resolution of these microscopic phenomena to maintain stability, resulting in prohibitive computational costs for large-scale or long-term simulations.
Two primary approaches have emerged to address these multi-physics challenges. 
The first approach employs domain decomposition, where distinct physical regimes are modeled separately in different regions of the simulated device (see \cite{klar1998asymptotic,bourgat1994coupling,arnold2001low}). 
While conceptually straightforward, this method requires sophisticated interface conditions to connect different models across regions and temporal stages. Researchers typically solve kinetic and fluid models simultaneously on separate subdomains, but this approach becomes particularly challenging in high-dimensional cases where interfaces may form complex curves or surfaces.
The second approach, known as asymptotic preserving (AP) schemes \cite{larsen1987asymptotic,golse1999convergence,jin1993fully,zhu2017vlasov,crouseilles2013asymptotic,degond2010asymptotic,blaustein2024structure}, offers a more unified framework. AP methods focus on developing numerical schemes that preserve the asymptotic transitions from microscopic to macroscopic models within a discrete framework. 
Unlike domain decomposition approaches, AP methods operate on a single set of microscopic equations, eliminating the need for explicit model coupling. 
These schemes automatically transition between microscopic and macroscopic solvers, providing a more seamless approach to multi-scale problems. 
For a comprehensive review of AP methodology and applications, we recommend \cite{Jin2022Asymptotic}.
This paper focuses specifically on the magnetized Vlasov--Poisson--Fokker--Planck (MVPFP) system. 
Unlike the standard Vlasov--Poisson--Fokker--Planck equation, which primarily concerns diffusion and high-field limits \cite{jin2011asymptotic,carrillo2021variational,Shi2022Contact,blaustein2024structure,jin2024asymptotic}, the MVPFP system introduces additional complexity through the Larmor radius, which becomes the dominant constraint on time step size in the presence of strong external magnetic fields.
This system has progressively attracted more attention in recent studies \cite{bostan2025asymptotic,bostan2023long}.
However, investigations into numerical algorithms for this equation are still lacking.
Therefore, the development of accurate and efficient algorithms for solving the MVPFP equations is of critical importance.

The structure of this paper is organized as follows. 
Section 2 presents the fundamental formulation and key properties of the MVPFP equation, followed by an analysis of its long-time asymptotic behavior under strong external magnetic fields. 
Section 3 introduces the stochastic particle-in-cell methodology, demonstrating its capability to preserve asymptotic properties. 
In Section 4, we develop time discretization schemes and rigorously establish their asymptotic-preserving characteristics, along with providing uniform estimates for the proposed methods. 
Numerical validations and practical applications are presented in Section 5 through comprehensive computational experiments. 
Finally, we conclude with a summary of our findings and contributions.

\section{The magnetized Vlasov--Poisson--Fokker--Planck equation}

\subsection{Description}
The magnetized Vlasov--Poisson--Fokker--Planck equation can be expressed as 
\begin{equation}\label{eqn:mvpfp}
\left\{
\begin{aligned}
&\partial_t f+v\cdot\nabla_x f+\left(E+v\times B\right)\cdot\nabla_v f=C(f),\\
&E=-\nabla_x \phi,\quad -\Delta_x{\phi}=\rho-\rho_0,\quad f(0,x,v)=f_0(x,v)
\end{aligned}
\right.
\end{equation}
where $f(t,x,v)$ is the distribution function of position $x\in\Omega_{x}\subset\mathbb{R}^{3}$ and velocity $v\in\Omega_{v}\subset\mathbb{R}^{3}$ at time $t$ while the Fokker-Planck type collision is
$$C(f)=\frac{1}{\tau}\nabla_v\cdot(\sigma\nabla_v f+vf),$$
where $\tau$ is the relaxation time and $\sigma$ is the velocity diffusion. Here, we assume that the ions form a uniform neutralizing background and $B$ is an external magnetic field.

For later use, we denote the $0$-th, $1$-st and $2$-nd moments of distribution function by 
$$\rho(t,x)=\int_{\Omega_{v}} f(t,x,v)dv,\ J(t,x)=\int_{\Omega_{v}}v f(t,x,v)dv$$
and $$S(t,x)=\int_{\Omega_{v}}v\otimes v f(t,x,v)dv=\int_{\Omega_{v}}vv^T f(t,x,v)dv.$$
In practical computation, the boundary conditions can be taken as zero, i.e., $f(t,x,v)=0$ and $\phi(x)=0$ on $\partial\Omega_x$ while the periodic boundary condition is also used frequently.
And we use $\Omega$ to denote the space $\Omega_x\times \Omega_v$ in the rest of this paper.

\begin{proposition}
MVPFP equation \eqref{eqn:mvpfp} has the following conservation laws:

Continuity equation:
\begin{equation}\label{ContinuityEq}
\partial_t \rho+\nabla_x\cdot J=0.
\end{equation}

Moment equation:
\begin{equation}\label{MomentEq}
\partial_t J+\nabla_x\cdot S-\rho E-\hat{B} J=-\frac{J}{\tau},
\end{equation}
where $\hat{B}=\left(
            \begin{array}{ccc}
              0 & b_3 & -b_2 \\
              -b_3 & 0 & b_1 \\
              b_2 & -b_1 & 0 \\
            \end{array}
          \right)$ with $B=(b_1,b_2,b_3)^T$.
\end{proposition}

\begin{proof}
By integrating the Vlasov equation in $v$ yield, one can get $LHS=\partial_t \rho+\nabla_x\cdot J$ and
\begin{align*}
RHS&=\int_{\Omega_v}C(f)dv\\
&=\frac{1}{\tau}\int_{\Omega_v}\nabla_v\cdot(\sigma\nabla_v f+vf)dv\\
&=\frac{1}{\tau}\left(\sigma\int_{\Omega_v}\Delta_v fdv+3\int_{\Omega_v}fdv+\int_{\Omega_v}v\cdot \nabla_v fdv\right)\\
&=\frac{1}{\tau}(3\rho-3\rho)=0.
\end{align*}

On the other hand, by multiplying the Vlasov equation by $v$ and integrating in $v$ leads to
$$LHS=\partial_t J+\nabla_x\cdot S-\rho E-\hat{B} J$$
and
\begin{align*}
RHS&=\int_{\Omega_v}vC(f)dv\\
&=\frac{1}{\tau}\left(\sigma\int_{\Omega_v}v\Delta_v fdv+3\int_{\Omega_v}vfdv+\int_{\Omega_v}v(v\cdot \nabla_v f)dv\right)\\
&=\frac{1}{\tau}(3J-4J)=-\frac{J}{\tau}.
\end{align*}
\end{proof}

The conservative properties of the MVPFP equation \eqref{eqn:mvpfp} can be described as the following proposition.
\begin{proposition}
If the potential function $\phi$ has a zero boundary or a periodic boundary, $f$ is periodic in $x$ and is
compactly supported in $v$, the following conservative properties hold.

(a) {\bf{Charge}}: $\mathcal{Q}=\int_{\Omega} fdxdv=\int_{\Omega_x} \rho dx$.
$$\frac{d\mathcal{Q}}{dt}=0.$$

(b) {\bf{Energy}}: $\mathcal{H}=\frac{1}{2}\int_{\Omega} \vert v\vert^2 fdxdv+\frac{1}{2}\int_{\Omega_x} \vert E\vert^2 dx$.
$$\frac{d\mathcal{H}}{dt}=\frac{1}{\tau}\left(3\sigma\mathcal{Q}-\int_{\Omega} \vert v\vert^2 fdxdv\right).$$

(c) {\bf{Entropy}}: $\mathcal{S}=\int_{\Omega} f\ln{f} dxdv$.
$$\frac{d\mathcal{S}}{dt}=\frac{1}{\tau}\left(3\mathcal{Q}-\sigma\int_{\Omega} \frac{\vert \nabla_v f\vert^2}{f} dxdv\right).$$
\end{proposition}

\begin{proof}
$\\ $
\par{(a): By integrating the continuity equation \eqref{ContinuityEq} in $x$ yield, one can directly get the result.}

(b): By multiplying the Vlasov equation by $\vert v\vert^2$ and integrating in $x$ and $v$ gives
$$\frac{d}{dt} \int_{\Omega}\vert v\vert^2 fdxdv+\int_{\Omega} \nabla_x\cdot(\vert v\vert^2 vf) dxdv+\int_{\Omega} \vert v\vert^2\nabla_v\cdot((E+v\times B)f) dxdv=\int_{\Omega} \vert v\vert^2 C(f) dxdv.$$
The second term on the left hand side vanishes due to $f$ is periodic in $x$.
By using Green formula, it yields
$$\int_{\Omega} \vert v\vert^2(\nabla_v\cdot Ef)dxdv=-2\int_{\Omega} (v\cdot E)fdxdv=-2\int_{\Omega_x}E\cdot Jdx.$$
Then by using the continuity equation \eqref{ContinuityEq} and Poisson equation, we can obtain
$$\int_{\Omega_x}\nabla_x\phi\cdot Jdx=-\int_{\Omega_x}\phi\nabla_x\cdot Jdx=\int_{\Omega_x}\phi\frac{\partial\rho}{\partial t}dx=\frac{1}{2}\frac{d}{dt}\int_{\Omega_x}\vert E\vert^2 dx.$$
Moreover, one can have $\int_{\Omega} \vert v\vert^2\nabla_v\cdot (v\times B)f dxdv=0$ due to $\nabla_v\cdot (v\times B)=0$.
Finally, we have 
$$LHS=\frac{d}{dt} \left(\int_{\Omega}\vert v\vert^2 fdxdv+\int_{\Omega_x}\vert E\vert^2 dx\right).$$
Also, it can be obtained that
\begin{align*}
RHS&=\frac{1}{\tau}\left(\sigma\int_{\Omega} \vert v\vert^2 \Delta_v fdxdv+3\int_{\Omega} \vert v\vert^2 fdxdv+\int_{\Omega} \vert v\vert^2 v\cdot \nabla_v fdxdv\right)\\
&=\frac{1}{\tau}\left(6\sigma\mathcal{Q}+3\int_{\Omega} \vert v\vert^2 fdxdv-5\int_{\Omega} \vert v\vert^2 fdxdv\right).
\end{align*}

(c): By multiplying the Vlasov equation by $\ln{f}+1$ and integrating in $x$ and $v$ results
$$LHS=\frac{d}{dt} \int_{\Omega}f\ln{f}dxdv+\int_{\Omega} (v\cdot\nabla_x f)\ln{f} dxdv+\int_{\Omega} ((E+v\times B)\cdot\nabla_v f)\ln{f} dxdv.$$
Similarly, the second and third integrals vanish since $f$ is periodic in $x$ and
compactly supported in $v$.
On the other hand, we have
\begin{align*}
RHS&=\frac{1}{\tau}\int_{\Omega}(\ln{f}+1)(\sigma\Delta_v f+3f+v\cdot\nabla_v f)dxdv\\
&=\frac{1}{\tau}\int_{\Omega}((\ln{f}+1)\sigma\Delta_v f+3f)dxdv\\
&=\frac{1}{\tau}\int_{\Omega}(\sigma\ln{f}\Delta_v f+3f)dxdv\\
&=\frac{1}{\tau}\int_{\Omega}\left(-\frac{\sigma}{f}\vert\nabla_v f\vert^2+3f\right)dxdv.
\end{align*}

\end{proof}

\subsection{Asymptotic behavior of the MVPFP equation with a strong external magnetic field}

In this paper, we focus on simulating the asymptotic behavior of the two-dimensional MVPFP equation.
And in order to study such equation, one should take a small parameter related to the magnetic field and the time scale which leads to the following system
\begin{equation}\label{eqn:scale}
\left\{
		\begin{aligned}
&\varepsilon\partial_t f+v\cdot\nabla_x f+\left(E+\frac{b(x)K}{\varepsilon}v\right)\cdot\nabla_v f=C(f)\,,\\\,\\
&-\Delta_x{\phi}=\rho-\rho_0,
\end{aligned}
		\right.
	\end{equation}
where $K=\left(
            \begin{array}{cc}
              0 & 1 \\
              -1 & 0
            \end{array}
          \right).$
In this case, $x\in\Omega_{x}\subset\mathbb{R}^{2},v\in\Omega_{v}\subset\mathbb{R}^{2}$ and it has a strong magnetic field, which is $B(x)=\frac{1}{\varepsilon}[0,0,b(x)]^T$.
For this equation, we can give its limit system by the following theorem.

\begin{theorem}
Assume that $E\in W^{1,\infty}$ and $b\in C^1(\Omega_x)$.
When $\varepsilon\to 0$, the MVPFP system \eqref{eqn:mvpfp} has a limit system:
\begin{equation}\label{eqn:limit}
    \partial_t\rho+\nabla_x\cdot\left[\rho\left(\frac{K E}{b(x)}-\sigma\frac{K\nabla b(x)}{b(x)^2}\right)\right]=0.
\end{equation}
\end{theorem}

\begin{proof}

Firstly, we give the limit system of the MVPFP system \eqref{eqn:mvpfp} at $\varepsilon\to 0$. Integrating the first equation of system \eqref{eqn:mvpfp} over $v$ in $\mathbb{R}^N$, one can get:
\begin{equation}\label{eqn:fm}
\varepsilon\partial_t \rho+\nabla_x \cdot J=0.
\end{equation}
Next, one multiplies the first equation of system \eqref{eqn:mvpfp} by $v$ and  integrates over $\Omega_{v}$ yielding:
\begin{equation}\label{eqn:sm}
 \varepsilon  \partial_t J+\nabla_x\cdot S-\rho E-\frac{b(x)K}{\varepsilon}J=-\frac{J}{\tau}.
\end{equation}
In order to get the leading order of term $\nabla_x\cdot S$, and recall the energy equation and the entropy equation:
\begin{equation}\label{eqn:es}
		\left\{
		\begin{aligned}
			&\varepsilon\frac{d\mathcal{H}}{dt}=-\frac{1}{\tau}\int_{\Omega}(\sigma\nabla_v f+vf)\cdot vdxdv\,,\\
			&\varepsilon\frac{d\mathcal{S}}{dt}=-\frac{1}{\tau}\int_{\Omega}(\sigma\nabla_v f+vf)\cdot\frac{\nabla_v f}{f}dxdv.
		\end{aligned}
		\right.
	\end{equation}
 and we can know that
\begin{equation}\label{eqn:em}
 \varepsilon \frac{d}{dt} \left(\mathcal{H}+\sigma \mathcal{S} \right)=-\frac{1}{\tau}\int_{\Omega}\frac{\vert \sigma\nabla_v f+vf\vert^2}{f}dxdv,
\end{equation}
which indicates that 
\begin{equation*}
 -\int_{\Omega}\frac{\vert \sigma\nabla_v f+vf\vert^2}{f}dxdv=\order(\varepsilon),
\end{equation*}
as $\varepsilon\to 0$, and we can get
\begin{equation}\label{eqn:tf}
 -\int_{\Omega}\frac{\vert \sigma\nabla_v f+vf\vert^2}{f}dxdv=0.
\end{equation}
On the other hand, equation \eqref{eqn:tf} can be rewritten as
\begin{equation}
 -\int_{\Omega}\frac{\vert \sigma W\nabla_v \left(f/W\right)\vert^2}{f}dxdv=0,
\end{equation}
with the Maxwellian equilibrium $W = \frac{1}{2\pi\sigma}{\rm exp}\left(-\frac{\vert v\vert^2}{2\sigma}\right)$, which indicates that
\begin{equation}\label{eqn:leado}
f=\rho W+\order(\varepsilon).
\end{equation}
It can be observed that $\int_{\Omega_v}(v\otimes v-\sigma I_2)fdv\to 0$ as $\varepsilon\to 0$ and $$\nabla_x\cdot\left(\frac{K\nabla_x\cdot \sigma\rho I_2}{b(x)}\right)=\nabla_x\cdot\left(\sigma\rho\frac{K\nabla_xb(x)}{b(x)^2}\right).$$
Then by substituting equation \eqref{eqn:leado} into equation \eqref{eqn:sm} and combining the resulting equation with equation \eqref{eqn:fm}, one can get
\begin{equation}
    \partial_t\rho+\nabla_x\cdot\left[\rho\left(\frac{K E}{b(x)}-\sigma\frac{K\nabla_x b(x)}{b(x)^2}\right)\right]=\order(\varepsilon),
\end{equation}
and the limit equation of MVPFP system is
\begin{equation*}
    \partial_t\rho+\nabla_x\cdot\left[\rho\left(\frac{K E}{b(x)}-\sigma\frac{K\nabla_x b(x)}{b(x)^2}\right)\right]=0.
\end{equation*}

\end{proof}
We can further obtain the following corollaries for two special cases.
\begin{corollary}
If $b(x)$ satisfies one of the following conditions,
\begin{itemize}
    \item {\bf{(Uniform magnetic field)}} $b(x)=b_0=Const\neq 0$,
    \item {\bf{(Maximal ordering scaling)}} $b(x)=\tilde{b}(\varepsilon x)$ and $\tilde{b}(0)=b_0=Const\neq 0$,
\end{itemize}
the limit equation will be 
\begin{equation}\label{eqn:limit-spe}
\partial_t\rho+\nabla_x\cdot\left(\rho\frac{K E}{b_0}\right)=0.
\end{equation}
Moreover, (\ref{eqn:limit-spe}) in particular can be further rewritten as
\begin{equation}\label{eqn:limit-homo}
\partial_t\rho+\frac{KE}{b_0}\cdot\nabla_x\rho=0.
\end{equation}
\end{corollary}

\section{Asymptotic-preserving stochastic particle-in-cell method}
In this section, we will construct the time and spatial discretization for the MVPFP system and show its asymptotic-preserving (AP) property. 

\subsection{Stochastic particle-in-cell method}
First, we use Particle-in-cell method to obtain a discrete distribution function which reads 
\begin{equation}\label{SPIC}
f_h(x,v,t)=\sum\limits_{s=1}^{N_p}\alpha_s\delta(x-X_s(t))\delta(v-V_s(t)).
\end{equation}
Here, $N_p$ is the number of particles, while $\alpha_s$ is the weight of the particle.

\begin{theorem}\label{SPICTH}
In the sense of expectation, the approximate distribution function $f_h$ of \eqref{eqn:scale} satisfying \eqref{SPIC} can be obtained by solving the following SDE system
\begin{equation}\label{SDE}
\left\{
\begin{aligned}
&dX_s=\frac{1}{\varepsilon}V_s dt,\\
&dV_s=\frac{1}{\varepsilon}\left(E(X_s)+\frac{b(X_s)}{\varepsilon}KV_s-\frac{1}{\tau}V_s\right)dt+\sqrt{\frac{2\sigma}{\varepsilon\tau}}d{\bf{W}},
\end{aligned}
\right.
\end{equation}
where ${\bf{W}}$ is $2$-dimensional Brownian motion.
\end{theorem}

\begin{proof}
First, we define the following functional
$$\langle f_h,\psi(t,x,v)\rangle:=\int_{[0,T]\times\Omega}f_h\psi(t,x,v) dxdvdt=\sum_{s=1}^{N_p}\alpha_s\int_0^T\psi(t,X_s,V_s)dt,$$
where $\psi$ is a test function satisfying $\psi\in C_{0}^{\infty}([0,T]\times\Omega)$.
Then we can further obtain
$$\left\langle\frac{\partial f_h}{\partial t},\psi(t,x,v)\right\rangle=-\sum_{s=1}^{N_p}\alpha_s
\int_0^T\frac{\partial\psi}{\partial t}(t,X_s,V_s)dt.$$
By taking the expectation of above equation on both sides, it leads to
$$\mathbb{E}\left(\left\langle \frac{\partial f_h}{\partial t},\psi\right\rangle\right)=-\sum_{s=1}^{N_p}\alpha_s
\int_0^T \mathbb{E}\left(\frac{\partial\psi}{\partial t}(t,X_s,V_s)\right)dt.$$
On the other hand, the It{\^o} formula provides
\begin{align*}
d\psi(t,X_s,V_s)&=\left(\frac{\partial\psi}{\partial t}+\frac{V_s}{\varepsilon}\cdot\frac{\partial\psi}{\partial X_s}+\frac{1}{\varepsilon}\left(E(X_s)+\frac{b(X_s)}{\varepsilon}KV_s-\frac{V_s}{\tau}\right)\cdot\frac{\partial\psi}{\partial V_s}+\frac{\sigma}{\varepsilon\tau}\Delta_{V_s}\psi\right)dt\\
&+\sqrt{\frac{2\sigma}{\varepsilon\tau}}\frac{\partial\psi}{\partial V_s}\cdot d{\bf{W}}.
\end{align*}
Since $\psi$ is compactly supported and $\mathbb{E}({\bf{W}})=0$, one can have
$$\mathbb{E}\left(\int_0^T\frac{\partial\psi}{\partial t}dt\right)=-\frac{1}{\varepsilon}\mathbb{E}\left(\int_0^T\left(V_s\cdot\frac{\partial\psi}{\partial X_s}+\left(E(X_s)+\frac{b(X_s)}{\varepsilon}KV_s-\frac{V_s}{\tau}\right)\cdot\frac{\partial\psi}{\partial V_s}+\frac{\sigma}{\tau}\Delta_{V_s}\psi\right)dt\right).$$
And finally we can get
$$\mathbb{E}\left(\left\langle \frac{\partial f_h}{\partial t},\psi\right\rangle\right)=-\frac{1}{\varepsilon}\mathbb{E}\left(\left\langle v\cdot\frac{\partial f_h}{\partial x}+\left(E(x)+\frac{b(x)}{\varepsilon}Kv\right)\cdot\frac{\partial f_h}{\partial v}-\frac{1}{\tau}\frac{\partial}{\partial v}\cdot(vf_h)-\frac{\sigma}{\tau}\Delta_v f_h,\psi\right\rangle\right).$$

\end{proof}

\subsection{Asymptotic behavior of \eqref{SDE} in maximal ordering case}
In this subsection, we discuss the asymptotic behavior of the characteristic line equation of the MVPFP equation in the maximal ordering case.
From Theorem~\ref{SPICTH}, it reads
\begin{equation}\label{SDE:maxorder}
\left\{
\begin{aligned}
&dx=\frac{1}{\varepsilon}v dt,\\
&dv=\frac{1}{\varepsilon}\left(E(x)+\frac{b(x)}{\varepsilon}Kv-\frac{1}{\tau}v\right)dt+\sqrt{\frac{2\sigma}{\varepsilon\tau}}d{\bf{W}},
\end{aligned}
\right.
\end{equation}
where $b(x)=\tilde{b}(\varepsilon x)$.
And we can get the following theorem.
\begin{theorem}\label{GCODE}
Assume that $b\in C^1(\Omega_x)$ and $\Omega_x$ is compact.
Then, in the limit $\varepsilon\to 0$, it follows that $x\to u$, where $u$ corresponds to the guiding center approximation
\begin{equation}\label{GCeq1}
\dot{u}=\frac{KE(u)}{b_0}=:\mathcal{R}_0E(u),\ u(0)=x(0),
\end{equation}
where $b_0=\tilde{b}(0)$.
Equation \eqref{GCeq1} is also the characteristic line equation of \eqref{eqn:limit-homo}.
\end{theorem}

\begin{proof}
By using the denotation $\mathcal{R}_{\varepsilon}:=\frac{b_0\tau^2K+\varepsilon\tau I_2}{b_0^2\tau^2+\varepsilon^2}$ and multiplying the second equation of \eqref{SDE:maxorder} by $\varepsilon\mathcal{R}_{\varepsilon}$, one can get
\begin{align*}
\varepsilon\mathcal{R}_{\varepsilon}dv&=\left(\mathcal{R}_{\varepsilon}E(x)-\frac{1}{\varepsilon}v\right)dt+\frac{b_0^2\tau^2I_2-b_0\tilde{b}(\varepsilon x)\tau^2I_2+\varepsilon\tau \tilde{b}(\varepsilon x)K-\varepsilon\tau b_0 K}{\varepsilon(b_0^2\tau^2+\varepsilon^2)}vdt+\sqrt{\frac{2\sigma\varepsilon}{\tau}}\mathcal{R}_{\varepsilon}d{\bf{W}}\\
&=\mathcal{R}_{\varepsilon}E(x)dt-dx+\frac{b_0^2\tau^2I_2-b_0\tilde{b}(\varepsilon x)\tau^2I_2+\varepsilon\tau \tilde{b}(\varepsilon x)K-\varepsilon\tau b_0 K}{(b_0^2\tau^2+\varepsilon^2)}dx+\sqrt{\frac{2\sigma\varepsilon}{\tau}}\mathcal{R}_{\varepsilon}d{\bf{W}}.
\end{align*}
Then SDE \eqref{SDE:maxorder} converges uniformly as $\varepsilon\to 0$ to the deterministic trajectory of \eqref{GCeq1}.
\end{proof}
Obviously, Theorem~\ref{GCODE} ensures that the particle method preserves the asymptotic property from micro-equations to macro-equations.
Thus, in this framework, we can further construct AP numerical algorithms.

\section{Time discretization of \eqref{SDE:maxorder}}

For the time discretization of \eqref{SDE:maxorder}, we combine the semi-implicit schemes and the Milstein method.

{\bf APSI1:}
By denoting $\delta=\frac{\Delta t}{\varepsilon}$ and $\lambda=\frac{\Delta t}{\varepsilon^2}$, a first order semi-implicit scheme is given by
\begin{equation}\label{APSI1}
\begin{aligned}
&x^{n+1}=x^n+\delta v^{n+1},\\
&v^{n+1}=v^n+\delta E(x^n)+\lambda \tilde{b}(\varepsilon x^n)K v^{n+1}-\frac{\delta}{\tau}v^{n+1}+\sqrt{\frac{2\sigma\delta}{\tau}}\xi_n,
\end{aligned}
\end{equation}
where $\xi_n\sim\mathcal{N}(0,1)$.
it also equals to
\begin{equation*}
\left(
\begin{array}{c}
x^{n+1}\\
v^{n+1}\\
\end{array}
\right)=
\left(
\begin{array}{cc}
I_2 & -\delta I_2\\
0 & (1+\frac{\delta}{\tau})I_2-\lambda \tilde{b}(\varepsilon x^n)K\\
\end{array}
\right)^{-1}
\left(
\begin{array}{c}
x^n\\
v^n+\delta E(x^n)+\sqrt{\frac{2\sigma\delta}{\tau}}\xi_n\\
\end{array}
\right).
\end{equation*}
For the upper triangular matrix above, its inverse matrix is easy to calculate according to the skew-symmetric matrix $K$.
Specifically, it is
\begin{equation*}
\left(
\begin{array}{cc}
I_2 & \delta \mathcal{M}(x^n)\\
0 &  \mathcal{M}(x^n)
\end{array}
\right),\ 
\end{equation*}
with 
$$
\mathcal{M}(y):=\frac{(1+\frac{\delta}{\tau})I_2+\lambda \tilde{b}(\varepsilon y)K}{(1+\frac{\delta}{\tau})^2+(\lambda \tilde{b}(\varepsilon y))^2}.
$$

{\bf APSI2:} Moreover, a second order scheme combined with Milstein method can be written as
\begin{equation}\label{APSI2}
\begin{aligned}
&x_n^{(1)}=x^n+\gamma\delta v_n^{(1)},\\
&v_n^{(1)}=v^n+\gamma\delta F_n^{(1)},\\
&x^{n+1}=x^n+(1-\gamma)\delta v_n^{(1)}+\gamma\delta v^{n+1},\\
&v^{n+1}=v^n+(1-\gamma)\delta F_n^{(1)}+\gamma\delta F_n^{(2)}+\sqrt{\frac{2\sigma\delta}{\tau}}\xi_n.
\end{aligned}
\end{equation}
where $\gamma=1-1/\sqrt{2}$ and
\begin{align*}
&x_n^{(2)}=x^n+\frac{\delta}{2\gamma}v_n^{(1)},\\
&F_n^{(1)}=E(x^n)+\frac{\tilde{b}(\varepsilon x^n)}{\varepsilon}Kv_n^{(1)}-\frac{1}{\tau}v_n^{(1)},\\
&F_n^{(2)}=E(x_n^{(2)})+\frac{\tilde{b}(\varepsilon x_n^{(2)})}{\varepsilon}K v^{n+1}-\frac{1}{\tau}v^{n+1}.
\end{align*}
For a vector form, it reads
\begin{equation*}
\begin{aligned}
&\left(
\begin{array}{c}
x_n^{(1)}\\
v_n^{(1)}\\
\end{array}
\right)=
\left(
\begin{array}{cc}
I_2 & -\gamma\delta I_2\\
0 & (1+\frac{\gamma\delta}{\tau})I_2-\gamma\lambda \tilde{b}(\varepsilon x^n)K
\end{array}
\right)^{-1}
\left(
\begin{array}{c}
x^n\\
v^n+\gamma\delta E(x^n)
\end{array}
\right),\\
&\left(
\begin{array}{c}
x^{n+1}\\
v^{n+1}\\
\end{array}
\right)=
\left(
\begin{array}{cc}
I_2 & -\gamma\delta I_2\\
0 & (1+\frac{\gamma\delta}{\tau})I_2-\gamma\lambda \tilde{b}(\varepsilon x_n^{(2)})K
\end{array}
\right)^{-1}
\left(
\begin{array}{c}
x^n+(1-\gamma)\delta v_n^{(1)}\\
v^n+(1-\gamma)\delta F_n^{(1)}+\gamma\delta E(x_n^{(2)})+\sqrt{\frac{2\sigma\delta}{\tau}}\xi_n
\end{array}
\right).
\end{aligned}
\end{equation*}

In the following part of this section, we present the asymptotic-preserving property of these algorithms.

\subsection{Uniform estimates of \eqref{APSI1}}
First of all, we have the following theorem for scheme \eqref{APSI1}.
\begin{theorem}\label{APTH}
Assume that $E\in W^{1,\infty}(\Omega_x)$, $b\in C^{1}(\Omega_x)$ and $\Omega_x$ is compact.
Let the solution $(x^N,v^N)$ to equation \eqref{APSI1} in finite time $T$ where $N\Delta{t}=T$ and $M_{\xi}:=\max\{\vert\xi_0\vert,\vert\xi_1\vert,\cdots,\vert\xi_{N-1}\vert\}$.
If $b(x)-b_0=\mathcal{O}(\varepsilon^s)$, there exists constants $C$ and $\lambda_0$ such that when $\lambda\ge\lambda_0$ the following estimate holds for a positive constant $\xi_0$
\begin{equation}\label{Estimate11}
\Vert x^N-u^N\Vert\lesssim\varepsilon^2\left(1+\Vert\varepsilon^{-1}v^0-\mathcal{R}(x^{0})E(x^{0})\Vert\right)+M_{\xi}\sqrt{\frac{2\sigma\varepsilon}{\tau\Delta{t}}}+\frac{\varepsilon^s\tau^2+\varepsilon^2+\varepsilon\tau}{\varepsilon^2+\tau^2},
\end{equation}
where $\mathcal{R}(y):=\frac{\tilde{b}(\varepsilon y)\tau^2K+\varepsilon\tau I_2}{(\tilde{b}(\varepsilon y)\tau)^2+\varepsilon^2}$ and $u^N$ is the numerical solution of the following guiding-center model
\begin{equation}\label{CUdis}
u^{n+1}=u^{n}+\Delta t \mathcal{R}_0E(u^{n}),\ u^0=x^0.
\end{equation}
\end{theorem}

\begin{proof}
Since $E\in W^{1,\infty}$ and $\Omega_x$ is compact, electric field and their first-order derivatives are all bounded.
Thus, a common upper bound $\kappa$ is taken for the estimates.
Firstly we introducing intermediate variables
\begin{equation}\label{znmid}
z^n=\varepsilon^{-1}v^n-\mathcal{R}(x^{n-1})E(x^{n-1}),\ n=1,\cdots,N.
\end{equation}
On the other hand, we have
$$v^1=\mathcal{M}(x^0)\left(v^0+\delta E(x^0)+\sqrt{\frac{2\sigma\delta}{\tau}}\xi\right).$$
By noticing that $\lambda\mathcal{M}(y)-\mathcal{R}(y)=-\mathcal{M}(y)\mathcal{R}(y)$, then it yields that 
$$z^1=\mathcal{M}(x^0)\left(\varepsilon^{-1}v^0-\mathcal{R}(x^{0})E(x^0)+\sqrt{\frac{2\sigma\delta}{\varepsilon^2\tau}}\xi\right).$$
And the following equality holds according to \eqref{znmid}
\begin{equation}\label{znplus}
z^{n+1}=\mathcal{M}(x^n)\left(z^n-\left(\mathcal{R}(x^{n})E(x^n)-\mathcal{R}(x^{n-1})E(x^{n-1})\right)+\sqrt{\frac{2\sigma\delta}{\varepsilon^2\tau}}\xi_n\right)
\end{equation}
for $\forall n\ge 1$.
Due to \eqref{APSI1}, it has
\begin{equation}\label{Cdis2}
x^n=x^{n-1}+\delta v^n=x^{n-1}+\Delta t z^n+\Delta t \mathcal{R}(x^{n-1})E(x^{n-1}).
\end{equation}
Let $\Lambda=\left\Vert\mathcal{M}(y)\right\Vert_{\infty}$, $a:=\Lambda\max\{1+2\kappa\Delta t,1+2\kappa^2\Delta t\}$, $b:=2\Lambda \kappa^4\Delta t$ and $c:=\Lambda\Vert\varepsilon^{-1}v^0-\mathcal{R}(x^{0})E(x^0)\Vert.$
Fixing the time step $\Delta t$, when $\varepsilon\to 0$, it has $\lambda\to\infty$.
Thus, there exist constants $\lambda_0,\ \alpha>0$ such that $a\leq \alpha<1$ when $\lambda\ge\lambda_0$.
After taking the norm on \eqref{znplus}'s both sides, it yields 
$$\Vert z^{n+1}\Vert\leq a^n c+\frac{1}{1-a}\left(b+M_{\xi}\sqrt{\frac{2\sigma\varepsilon}{\tau\Delta{t}}}\right).$$
By denoting $e^n:=x^n-u^n$, it can be obtained that
\begin{equation*}
\Vert e^n \Vert\leq \frac{a}{\Lambda}\Vert e^{n-1} \Vert+\Delta t a^{n-1}c+\Delta t\frac{b}{1-a}+\frac{M_{\xi}}{1-a}\sqrt{\frac{2\sigma\varepsilon\Delta{t}}{\tau}}+\Delta{t}\frac{\varepsilon^s\tau^2+\varepsilon^2+\varepsilon\tau}{\varepsilon^2+\tau^2},
\end{equation*}
according to
\begin{equation}\label{Rdiff}
\mathcal{R}(y)-\mathcal{R}_0\lesssim\frac{\varepsilon^s\tau^2+\varepsilon^2+\varepsilon\tau}{\varepsilon^2+\tau^2}.
\end{equation}
Finally, (\ref{Estimate11}) holds due to $e^0=0$.

\end{proof}

It can be found from Theorem~\ref{APTH} that the strong magnetic field suppresses the randomness of numerical solutions, thereby ensuring the asymptotic-preserving property of the algorithm.

\begin{corollary}
If the assumptions in theorem (\ref{APTH}) hold, it has the following estimate:
\begin{equation}\label{Estimate12}
\Vert\mathbb{E}(x^N)-u^N\Vert\lesssim\varepsilon^2\left(1+\Vert\varepsilon^{-1}v^0-\mathcal{R}(x^{0})E(x^{0})\Vert\right)+\frac{\varepsilon^s\tau^2+\varepsilon^2+\varepsilon\tau}{\varepsilon^2+\tau^2}.
\end{equation}
\end{corollary}

\begin{proof}
It's obvious that (\ref{Estimate12}) holds since $\xi_n\sim \mathcal{N}(0,1)$.
\end{proof}

Furthermore, by eliminating the error appearing at \eqref{Rdiff}, we can attain a better estimate.
\begin{corollary}
If the assumptions in theorem (\ref{APTH}) hold, it has the following estimate:
\begin{equation}
\Vert\mathbb{E}(x^N)-u^N\Vert\lesssim\varepsilon^2\left(1+\Vert\varepsilon^{-1}v^0-\mathcal{R}(x^{0})E(x^{0})\Vert\right),
\end{equation}
where $u^N$ is the numerical solution of the following model
\begin{equation}\label{GCeq2}
\dot{u}=\mathcal{R}(u)E(u),\ u(0)=x(0),
\end{equation}
with algorithm
\begin{equation*}
u^{n+1}=u^{n}+\Delta t \mathcal{R}(u^{n})E(u^{n}),\ u^0=x^0.
\end{equation*}
\end{corollary}

\begin{corollary}\label{orederTH}
Under the assumptions of Theorem~\ref{APTH}, the proposed scheme \eqref{APSI1} is of weak order $1$ when $\varepsilon\ll 1$. More precisely, it has
\begin{equation}\label{Estimateoreder}
\Vert \mathbb{E}(\phi(x^N))-\mathbb{E}(\phi(x(T)))\Vert\lesssim \Delta{t},
\end{equation}
for all $\phi\in C_0^4$ and $T=N\Delta{t}$ when $\varepsilon\ll 1$.
\end{corollary}

\begin{proof}

By the triangle inequality, we can obtain that
$$\Vert \mathbb{E}(x^N)-\mathbb{E}(x(T))\Vert\le\Vert \mathbb{E}(x^N)-u^N\Vert+\Vert u^N-u(T)\Vert+\Vert u(T)-\mathbb{E}(x(T))\Vert.$$
According to Theorem~\ref{GCODE} and Theorem~\ref{APTH},  we have
$$\Vert \mathbb{E}(x^N)-u^N\Vert+\Vert u(T)-\mathbb{E}(x(T))\Vert\lesssim \sqrt{\frac{\sigma\varepsilon}{\tau}}+\frac{\varepsilon^s\tau^2+\varepsilon^2+\varepsilon\tau}{\varepsilon^2+\tau^2}+\frac{\varepsilon^s\tau^2+\varepsilon^{s+1}\tau}{\varepsilon^2+\tau^2}+\varepsilon.$$
Then \eqref{Estimateoreder} holds due to \eqref{CUdis}.

\end{proof}

\subsection{Uniform estimates of \eqref{APSI2}}
For scheme (\ref{APSI2}), it follows: 
\begin{theorem}\label{APTH2}
Under the assumptions of theorem (\ref{APTH}), the estimate (\ref{Estimate11}) holds for $(x^N,v^N)$ to equation (\ref{APSI2}) while $u^N$ is the numerical solution of the guiding-center model
\begin{equation}\label{CUdis2}
u^{n+1}=u^{n}+(1-\gamma)\Delta{t}\mathcal{R}_0E(u^{n})+\gamma\Delta{t}\mathcal{R}_0E(u^{(1)}),
\end{equation}
where
$$u^{(1)}=u^{n}+\frac{\Delta t}{2\gamma}\mathcal{R}_0E(u^{n}),\ u^0=x^0.$$
\end{theorem}

\begin{proof}

The proof is similar to Theorem~\ref{APTH}, we briefly present the key points.
First, we define
\begin{equation*}
z^n_{1}=\varepsilon^{-1}v_{n-1}^{(1)}-\mathcal{R}(x^{n-1})E(x^{n-1}),\ z^n_{2}=\varepsilon^{-1}v^n-\mathcal{R}(x_{n-1}^{(2)})E(x_{n-1}^{(2)}).
\end{equation*}
Then for scheme (\ref{APSI2}) due to $\gamma\lambda\mathcal{M}_{\gamma}(y)-\mathcal{R}(y)=-\mathcal{M}_{\gamma}(y)\mathcal{R}(y)$, it follows that 
\begin{align*}
z^{n+1}_{1}=&\mathcal{M}_{\gamma}(x^{n})\left(z_2^{n}+\mathcal{R}(x_{n-1}^{(2)})E(x_{n-1}^{(2)})-\mathcal{R}(x^{n})E(x^{n})\right),\\ 
z^{n+1}_{2}=&\mathcal{M}_{\gamma}(x_n^{(2)})\left(z_2^n+\mathcal{R}(x_{n-1}^{(2)})E(x_{n-1}^{(2)})-\mathcal{R}(x_{n}^{(2)})E(x_{n}^{(2)})+\lambda(1-\gamma)\left(\tilde{b}(\varepsilon x^n)K-\frac{\varepsilon}{\tau}I_2\right)z_1^{n+1}+\sqrt{\frac{2\sigma\delta}{\varepsilon^2\tau}}\xi_n\right),
\end{align*}
where 
$$
\mathcal{M}_{\gamma}(y):=\frac{(1+\frac{\gamma\delta}{\tau})I_2+\gamma\lambda \tilde{b}(\varepsilon y)K}{(1+\frac{\gamma\delta}{\tau})^2+(\gamma\lambda\tilde{b}(\varepsilon y))^2}.
$$
By denoting $\Lambda=\left\Vert\mathcal{M}_{\gamma}(y)\right\Vert_{\infty}$, it has
\begin{equation}\label{errz12}
\begin{aligned}
&\Vert z^{n+1}_{1}\Vert\lesssim\Lambda(\Vert z^{n}_{2}\Vert+\Delta t\Vert z^{n}_{1}\Vert+\Delta t),\\ 
&\Vert z^{n+1}_{2}\Vert\lesssim\Lambda\left(\left(\frac{\Delta t}{2\gamma}+(1-\gamma)\lambda\right)\Vert z^{n+1}_{1}\Vert+\Delta t\Vert z^{n}_{1}\Vert+\Vert z^{n}_{2}\Vert+\Delta t+M_{\xi}\sqrt{\frac{2\sigma\delta}{\varepsilon^2\tau}}\right),
\end{aligned}
\end{equation}
due to
\begin{align*}
&\Vert x^{(2)}_{n-1}-x^n\Vert=\delta\Vert v_{n-1}^{(1)}-\gamma v^n\Vert\lesssim\Delta t(\Vert z^{n}_{1}\Vert+\gamma\Vert z^{n}_{2}\Vert+1),\\ 
&\Vert x^{(2)}_{n-1}-x^{(2)}_{n}\Vert\lesssim\Delta t\left(\frac{1}{2\gamma}\Vert z^{n+1}_{1}\Vert+\Vert z^{n}_{1}\Vert+\gamma\Vert z^{n}_{2}\Vert+1\right).
\end{align*}
Then (\ref{errz12}) leads to
\begin{equation*}
\Vert z^{n+1}_{1}\Vert+\Vert z^{n+1}_{2}\Vert\lesssim\Lambda\left(\Vert z^{n}_{1}\Vert+\Vert z^{n}_{2}\Vert+\Delta t+M_{\xi}\sqrt{\frac{2\sigma\delta}{\varepsilon^2\tau}}\right).
\end{equation*}
On the other hand, by taking difference between $x^{n+1}$ and $u^{n+1}$, (\ref{Estimate11}) holds due to
\begin{equation*}
\Vert e^{n+1}\Vert\lesssim\Vert e^{n}\Vert+\Delta t\Vert z^{n+1}_{1}\Vert+\Delta t\Vert z^{n+1}_{2}\Vert+\Delta{t}\frac{\varepsilon^s\tau^2+\varepsilon^2+\varepsilon\tau}{\varepsilon^2+\tau^2}
\end{equation*}
and
\begin{align*}
&z^1_{1}=\mathcal{M}_{\gamma}(x^{0})\left(\varepsilon^{-1}v^0-\mathcal{R}(x^{0})E(x^{0})\right),\\ 
&z^1_{2}=\mathcal{M}_{\gamma}(x_0^{(2)})\left(\varepsilon^{-1}v^0-\mathcal{R}(x_{0}^{(2)})E(x_{0}^{(2)})+\lambda(1-\gamma)\left(\tilde{b}(\varepsilon x^0)K-\frac{\varepsilon}{\tau}I_2\right)z_1^{1}+\sqrt{\frac{2\sigma\delta}{\varepsilon^2\tau}}\xi_0\right).
\end{align*}
    
\end{proof}

\begin{corollary}\label{orederTH2}
Under the assumptions of Theorem~\ref{APTH}, the proposed scheme \eqref{APSI2} is of weak order $2$ when $\varepsilon\ll 1$. More precisely, it has
\begin{equation}\label{Estimateoreder2}
\Vert \mathbb{E}(\phi(x^N))-\mathbb{E}(\phi(x(T)))\Vert\lesssim \Delta{t}^2,
\end{equation}
for all $\phi\in C_0^4$ and $T=N\Delta{t}$ when $\varepsilon\ll 1$.
\end{corollary}

\section{Numerical experiments}

This section is devoted to present the numerical results from the proposed numerical schemes for verifying the analysis results. 
By simulating the motion of a single particle, the asymptotic-preserving properties of our schemes can be observed and are coincide with our theoretical results.
Then an application of a $2+2$-dimensional MVPFP system is presented to display the advantage of our algorithm.
Moreover, in this numerical experiment, we have used the finite element method in~\cite{GAJ2022Hamiltonian} for spatial discretization.

\subsection{Benchmark problem}
First, a numerical benchmark test of the charged particle system \eqref{SDE:maxorder} is performed to show the efficiency of our proposed algorithms and verify estimates.
We consider the electromagnetic field with $E(x)=-x,b(x)=1+\varepsilon\sin(\sqrt{x_1^2+x_2^2}),$ and the initial data $x(0)=(0.3,0.2)^T,v(0)=(-0.7,0.08)^T$ if not otherwise specified.

\begin{figure}[h!]
\centering
\subfigure[]{
\includegraphics[scale=.45]{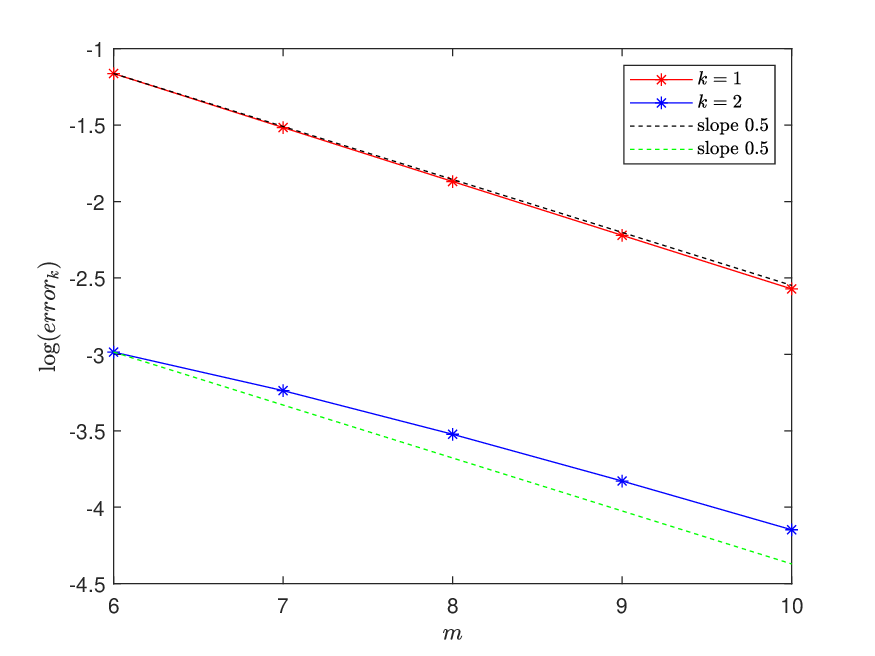}
}
\quad
\subfigure[]{
\includegraphics[scale=.45]{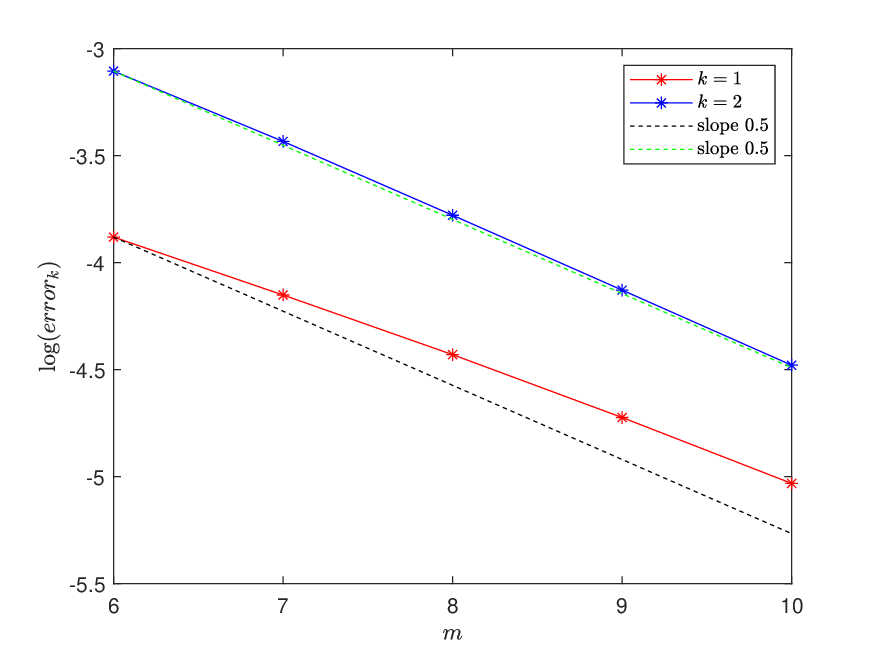}
}
\subfigure[]{
\includegraphics[scale=.45]{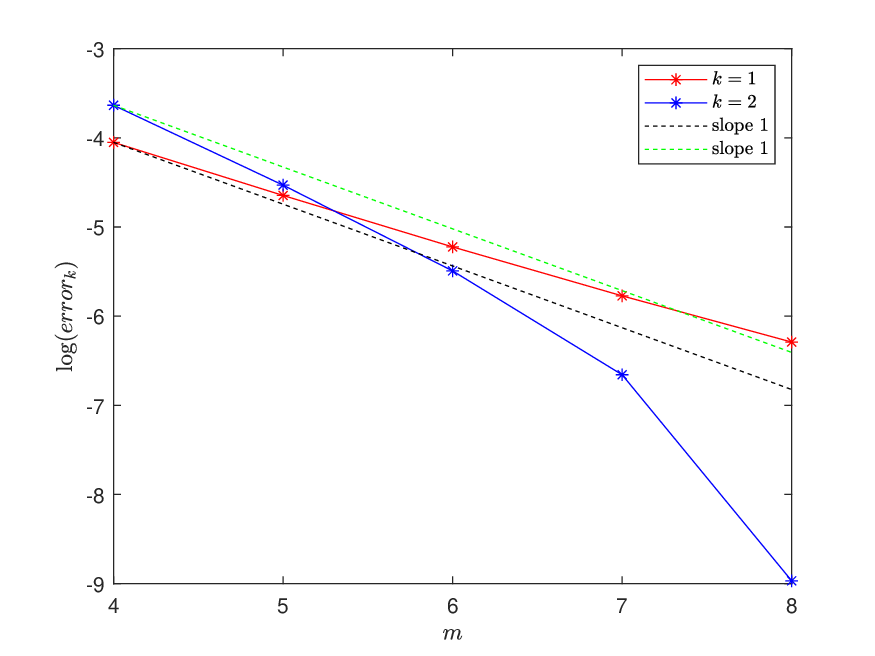}
}
\quad
\subfigure[]{
\includegraphics[scale=.45]{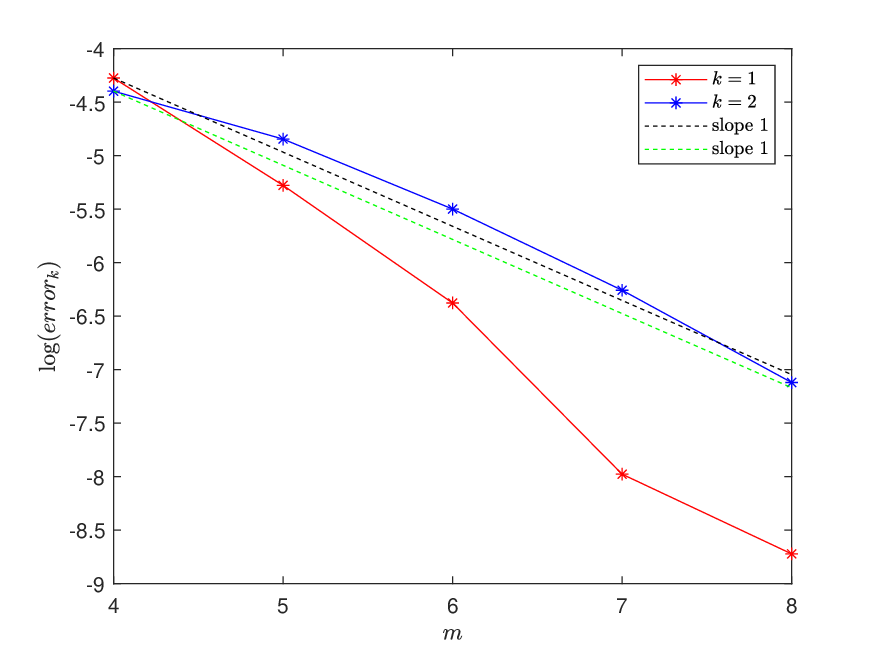}
}
\caption{Estimates $error_k=\vert x_k^N-u_k^N\vert$ of semi-implicit schemes with time step $\Delta t=\pi/30$ for $\varepsilon=2^{-m}$ and  $T=\pi$. (a)(c): APSI1 (\ref{APSI1}). (b)(d): APSI2 (\ref{APSI2}). (a)(b): $\sigma=\tau=1$. (c)(d): $\sigma=2^{-6},\tau=2^{6}$ (i.e. $\sigma\le\varepsilon\tau$).}
\label{img1}
\end{figure}

\begin{figure}[h!]
\centering
\subfigure[]{
\includegraphics[scale=.45]{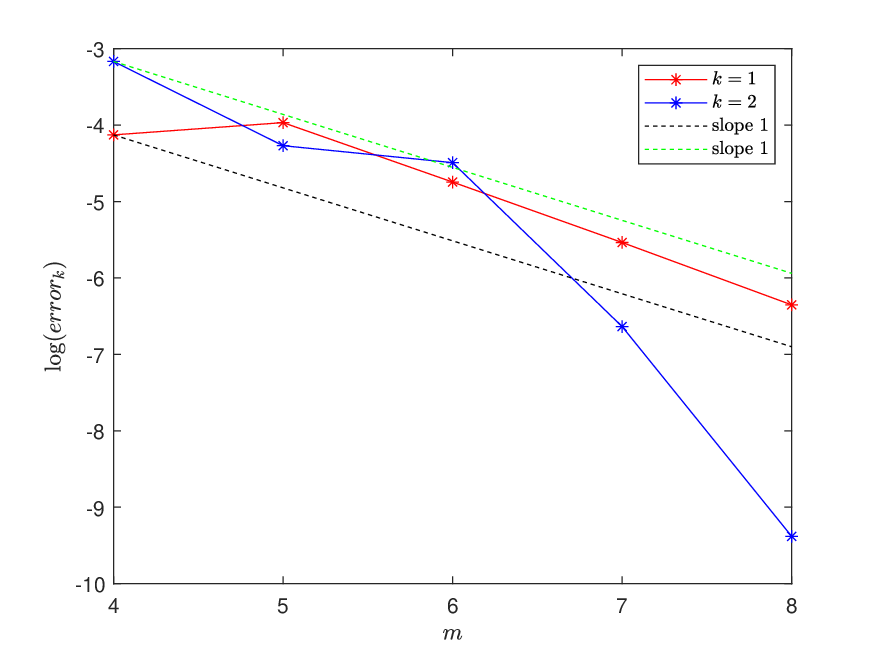}
}
\quad
\subfigure[]{
\includegraphics[scale=.45]{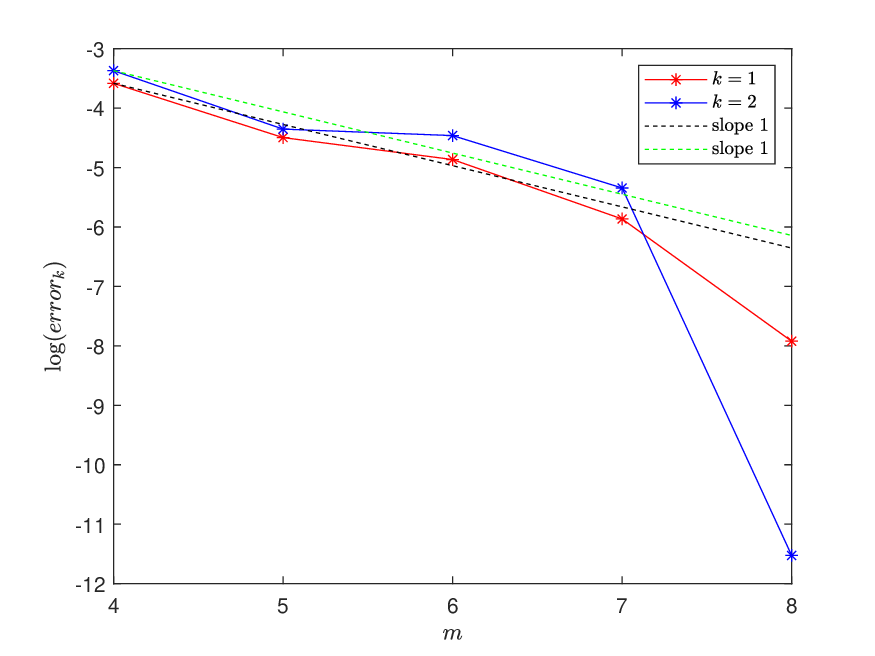}
}
\subfigure[]{
\includegraphics[scale=.45]{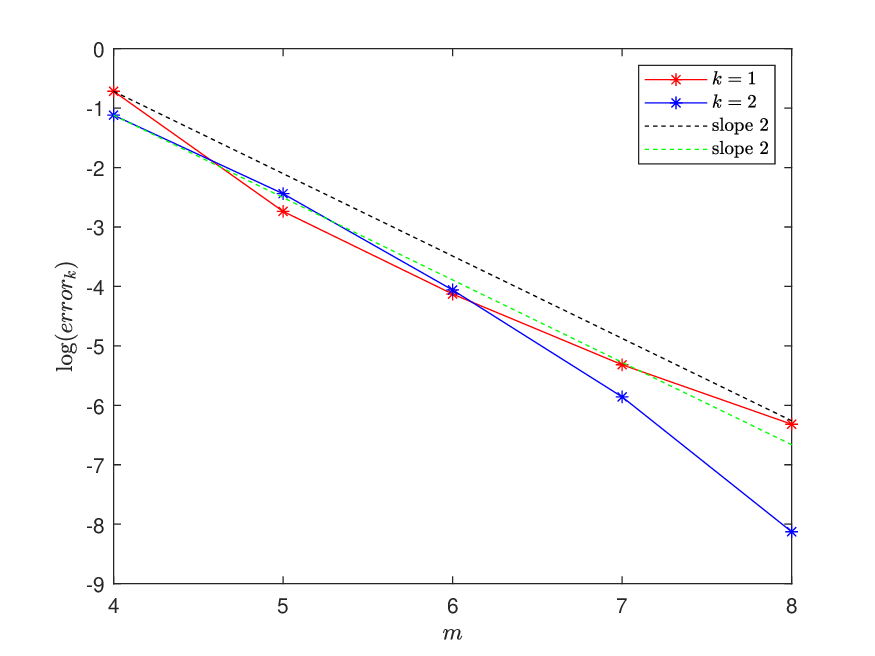}
}
\quad
\subfigure[]{
\includegraphics[scale=.45]{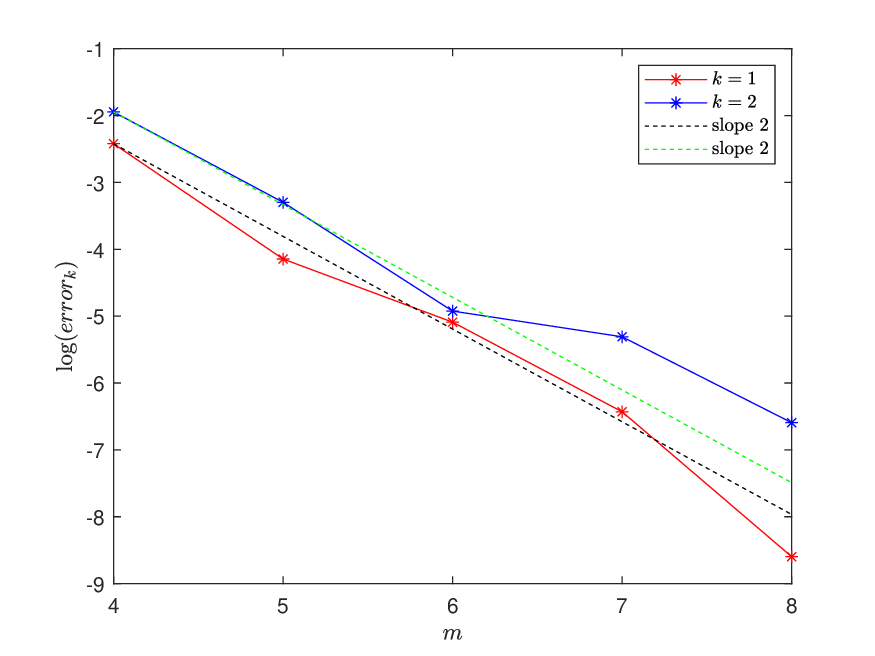}
}
\caption{Estimates $error_k=\vert\mathbb{E}(x_k^N)-u_k^N\vert$ of semi-implicit schemes with time step $\Delta t=\pi/30$ for $\varepsilon=2^{-m}$ at final time $T=\pi$ for $10^4$ Brownian paths when $\sigma=\tau=1$. (a)(c): APSI1 (\ref{APSI1}). (b)(d): APSI2 (\ref{APSI2}). (a)(b): $u^N$ is the numerical solution of \eqref{GCeq1}. (c)(d): $x(0)=(10,14)^T,v(0)=\mathcal{O}(\varepsilon)$, $u^N$ is the numerical solution of \eqref{GCeq2}.}
\label{img2}
\end{figure}

\begin{figure}[h!]
\centering
\subfigure[]{
\includegraphics[scale=.45]{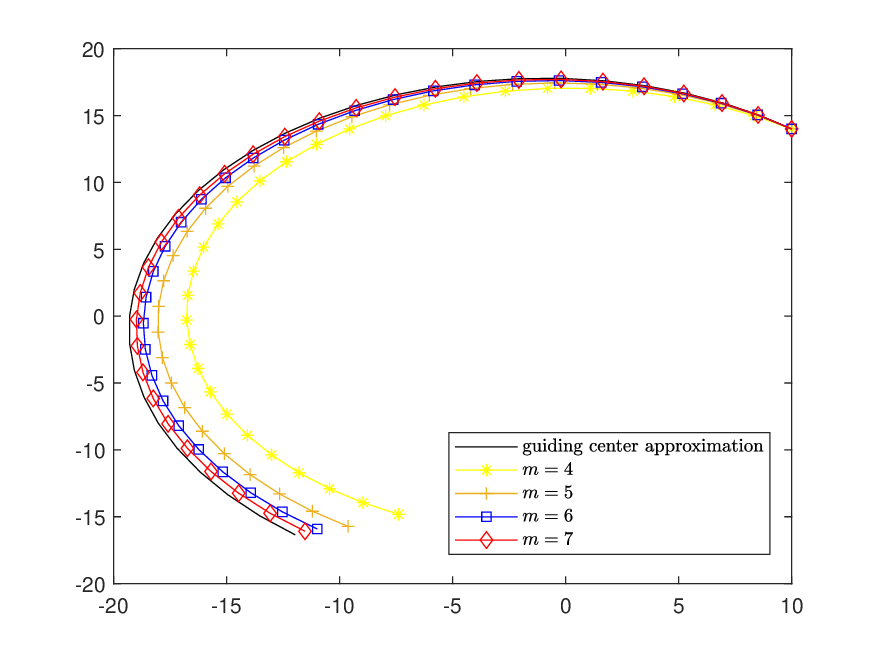}
}
\quad
\subfigure[]{
\includegraphics[scale=.45]{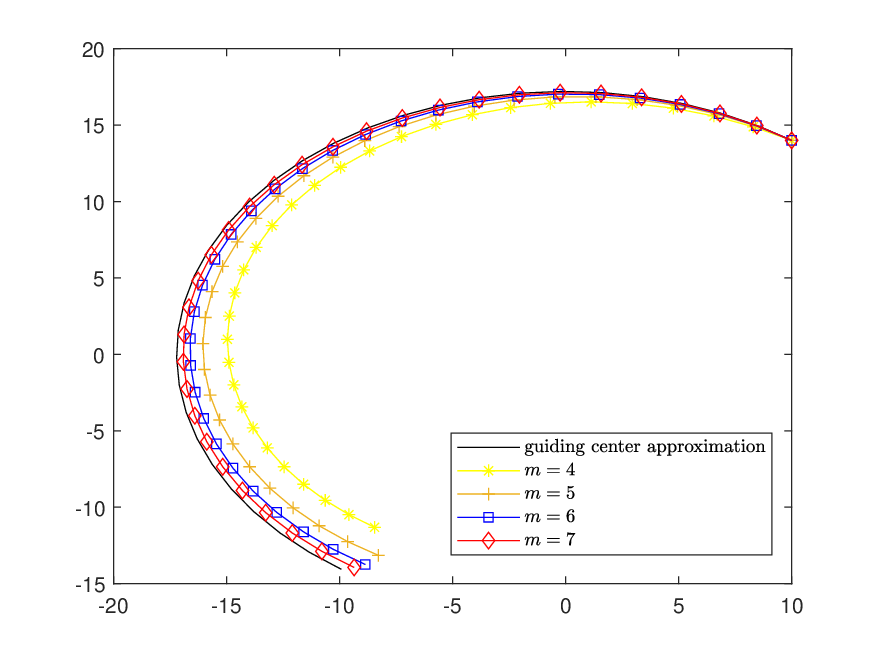}
}
\caption{Comparison between the solution of the guiding center model and the expectations of the solution solving by proposed algorithms with time step $\Delta t=\pi/30$ for $\varepsilon=2^{-m}$ at final time $T=\pi$ for $10^4$ Brownian paths when $\sigma=\tau=1$, $x(0)=(10,14)^T$. (a): APSI1 (\ref{APSI1}). (b): APSI2 (\ref{APSI2}).}
\label{img3}
\end{figure}

In Figure~\ref{img1}, we plot the estimates $error_k=\vert x_k^N-u_k^N\vert$ of the proposed schemes with time step $\Delta t=\pi/30$ for $\varepsilon=2^{-m}$ at the final time $T=\pi$.
It can be observed that the numerical errors match well with our theoretical estimations in Theorem~\ref{APTH}.
When the collision is small, their approximate order with the limit model can reach first order.
Then in Figure~\ref{img2}, we plot the estimates $error_k=\vert\mathbb{E}(x_k^N)-u_k^N\vert$.
After modifying the initial data $v(0)=\mathcal{O}(\varepsilon)$, it leads to $\Vert\mathbb{E}(x_k^N)-u_k^N\Vert=\mathcal{O}(\varepsilon^{2})$ due to $\Vert\varepsilon^{-1}v^0-\mathcal{R}(x^{0})E(x^{0})\Vert=\mathcal{O}(1)$ where $u^N$ is the numerical solution of \eqref{GCeq2}.
The figures are consistent with the results of our corollaries.
The trajectories calculating by numerical methods are shown in Figure~\ref{img3}.
When $\varepsilon\to 0$, the trajectory tends towards the numerical solution of the guiding center system.

\begin{figure}[h!]
\centering
\subfigure[]{
\includegraphics[scale=.45]{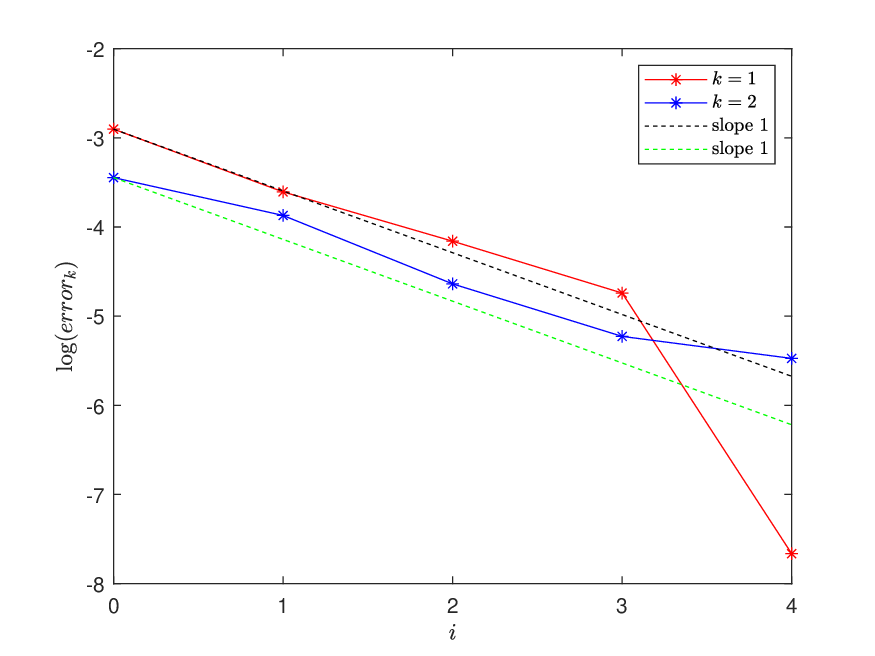}
}
\quad
\subfigure[]{
\includegraphics[scale=.45]{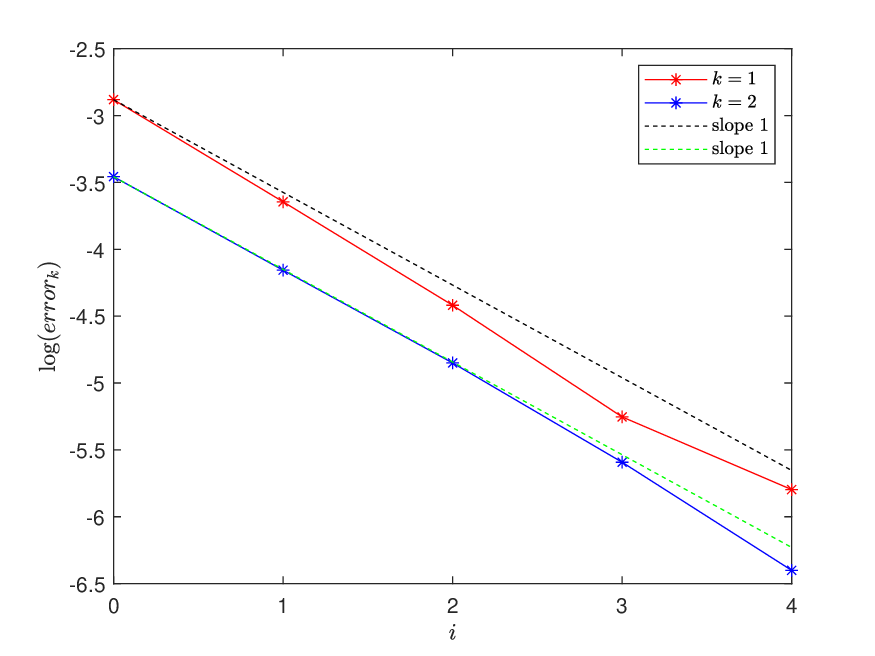}
}
\caption{Estimates $error_k=\vert \mathbb{E}(x_k^N)-\mathbb{E}(x_k(T))\vert$ of semi-implicit scheme \eqref{APSI1} with time step $\Delta t=\frac{\pi}{30}2^{-m}$ at final time $T=\pi$ for $10^4$ Brownian paths when $\sigma=\tau=1$. (a): $\varepsilon=10^{-2}$. (b): $\varepsilon=10^{-4}$.}
\label{img4}
\end{figure}

\begin{figure}[h!]
\centering
\subfigure[]{
\includegraphics[scale=.45]{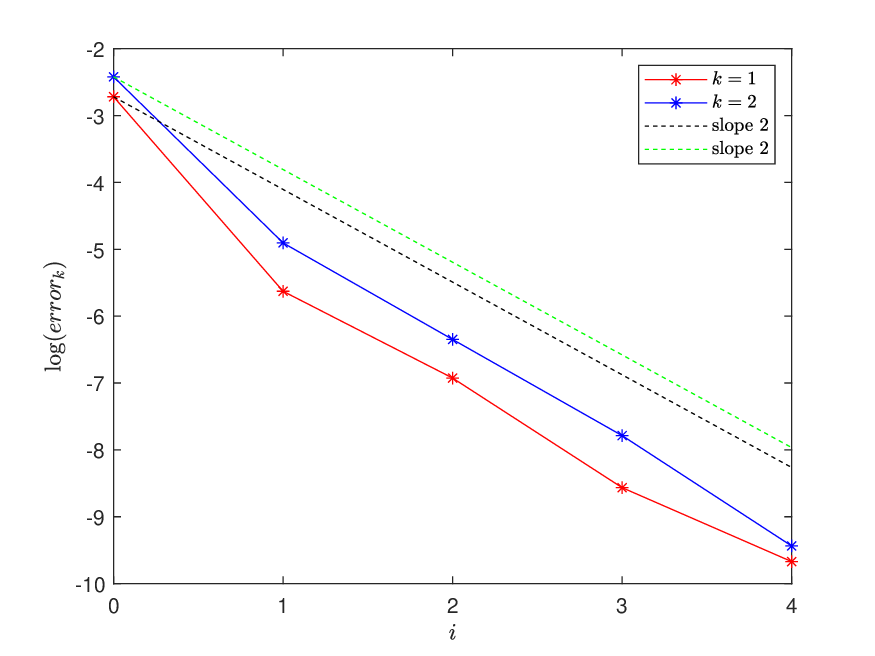}
}
\quad
\subfigure[]{
\includegraphics[scale=.45]{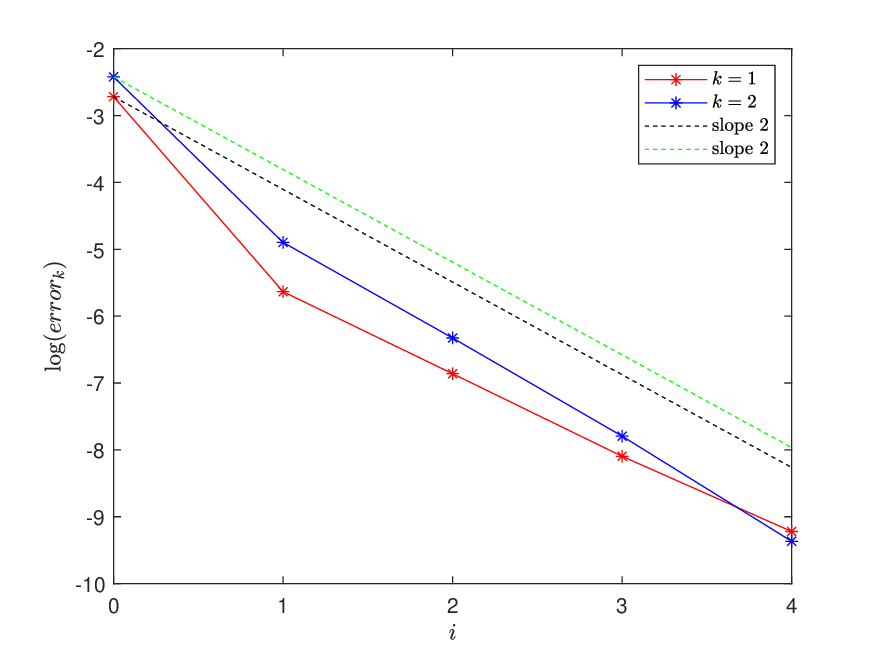}
}
\caption{Estimates $error_k=\vert \mathbb{E}(x_k^N)-\mathbb{E}(x_k(T))\vert$ of semi-implicit scheme \eqref{APSI2} with time step $\Delta t=\frac{2\pi}{15}2^{-m}$ at final time $T=\pi$ for $10^4$ Brownian paths when $\sigma=\tau=1$. (a): $\varepsilon=10^{-6}$. (b): $\varepsilon=10^{-8}$.}
\label{img5}
\end{figure}

Finally, we verify the convergence order of our methods when $\varepsilon\ll 1$ in Figure~\ref{img4} and Figure~\ref{img5}. 
Our methods \eqref{APSI1} and \eqref{APSI2} have all reached theoretical orders.

\subsection{Diocotron instability}
The diocotron instability is a plasma instability that occurs when two sheets of charges slip past each other. This stability is usually described by the guiding center model~\cite{Levy1965,Filbet2016,Ameres2018} and has been extensively simulated~\cite{Nicolas2014,Filbet2016,Filbet2018,Ameres2018,Chartier2020,GAJ2022Hamiltonian}. 
Without the collision term, there have been many numerical methods for studying this instability. Among them, the asymptotic-preserving schemes are introduced in~\cite{Filbet2018}, and the uniformly accurate methods are introduced in~\cite{Chartier2020}.
Here, we consider the model of two dimensional Vlasov--Poisson-–Fokker–-Planck system together with an external magnetic field. 
This system can be taken to model the instability which is usually encountered in magnetic fusion devices such as tokamaks. 
In this example, we also consider the effect of the collision on the instability.

We take the initial distribution function as
\begin{equation*}
f_0(x,v)=\frac{d_0(x)}{2\pi}\exp\left(-\frac{\vert v\vert^2}{2}\right),
\end{equation*}
where the initial density is
\begin{equation*}
d_0(x)=
\begin{cases}
(1+{\alpha}\cos(l\theta))\exp(-4(\vert x\vert-5)^2)&\mbox{if $r^-\leq\vert x\vert\leq{r}^+$,} \\
0&\mbox{otherwise,}
\end{cases}
\end{equation*}
with $\theta={\rm{atan}}(y/x)$ and $l$ the number of vortices. 
In our simulation, we take $r^-=3.5,r^+=6.5,\alpha=0.2$ and more than $6\times10^7$ particles for PIC simulation.
For the computational domain, we choose $\Omega_x=[-8,8]\times[-8,8]$ and $\Omega_v=[-4,4]\times[-4,4]$.
Furthermore, we use \eqref{APSI1} for time discretization. 

\begin{figure}[!ht]
\centering
\subfigure{
\includegraphics[scale=.4]{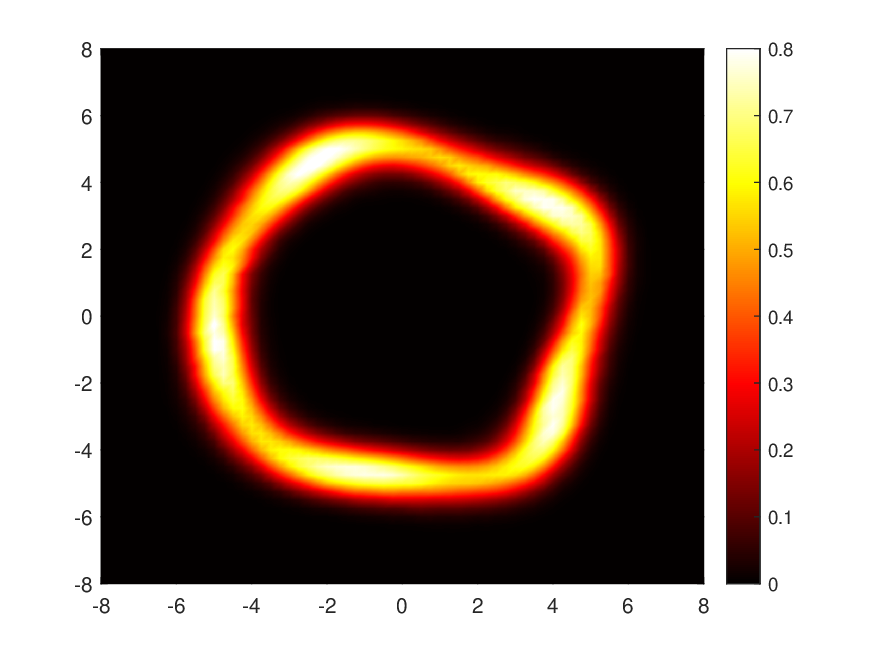}
}
\quad
\subfigure{
\includegraphics[scale=.4]{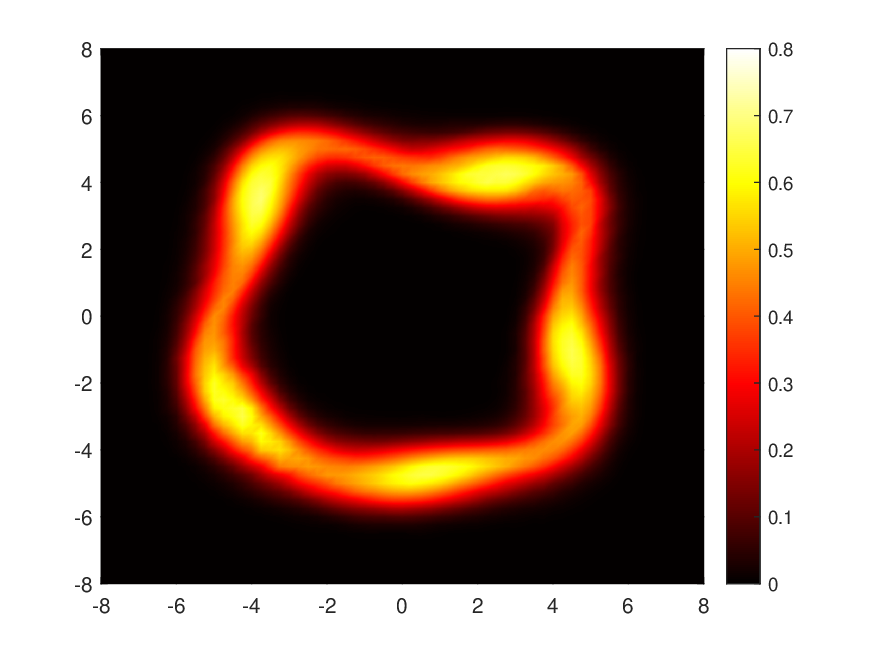}
}
\quad
\subfigure{
\includegraphics[scale=.4]{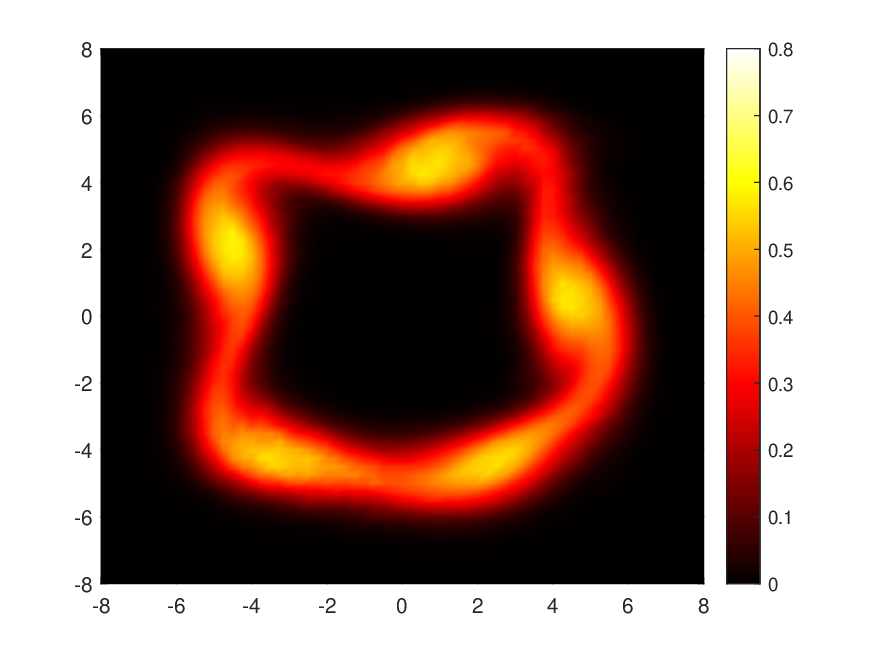}
}
\quad
\subfigure{
\includegraphics[scale=.4]{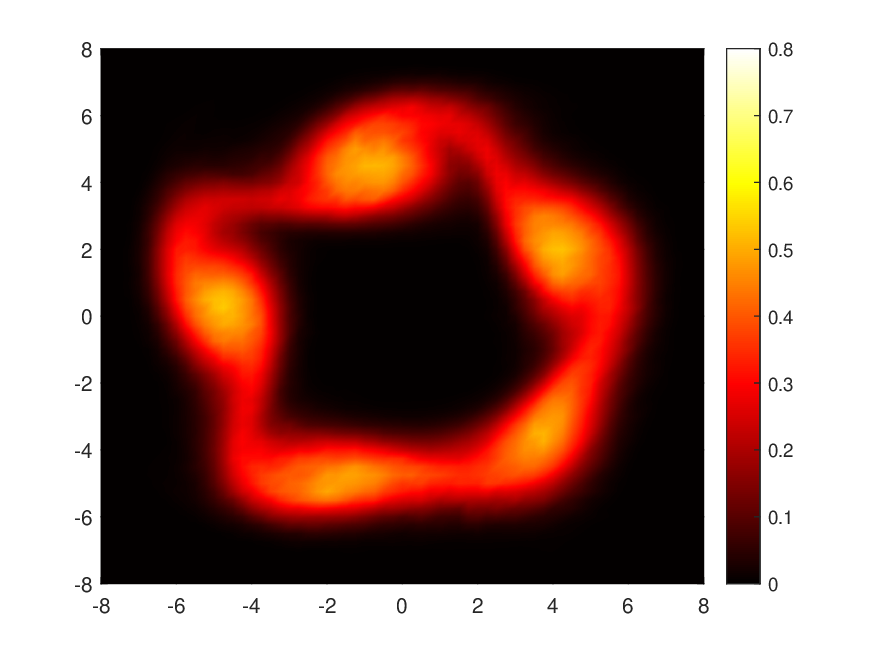}
}
\caption{Time evolution of the density $\rho$ for time $t=5,10,15,20$ and $l=5$ with parameter $\sigma=\tau=1,\varepsilon=10^{-2}$.}
\label{imgdi1}
\end{figure}

\begin{figure}[!ht]
\centering
\subfigure{
\includegraphics[scale=.4]{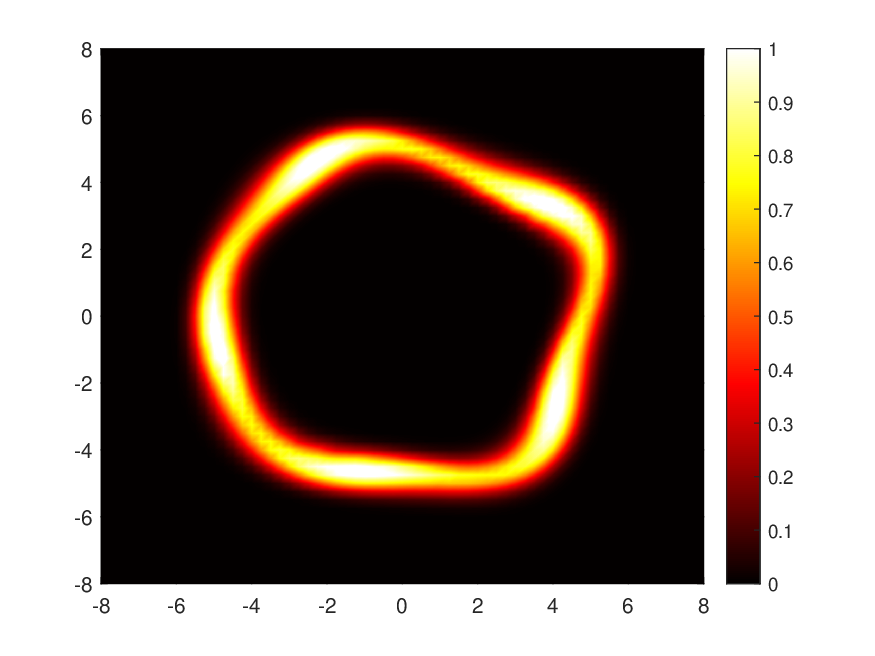}
}
\quad
\subfigure{
\includegraphics[scale=.4]{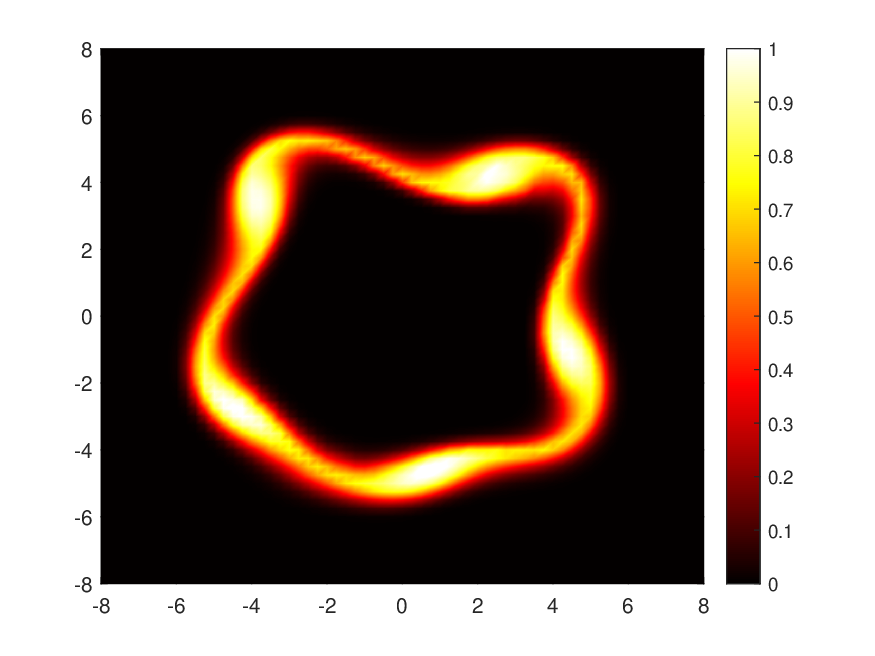}
}
\quad
\subfigure{
\includegraphics[scale=.4]{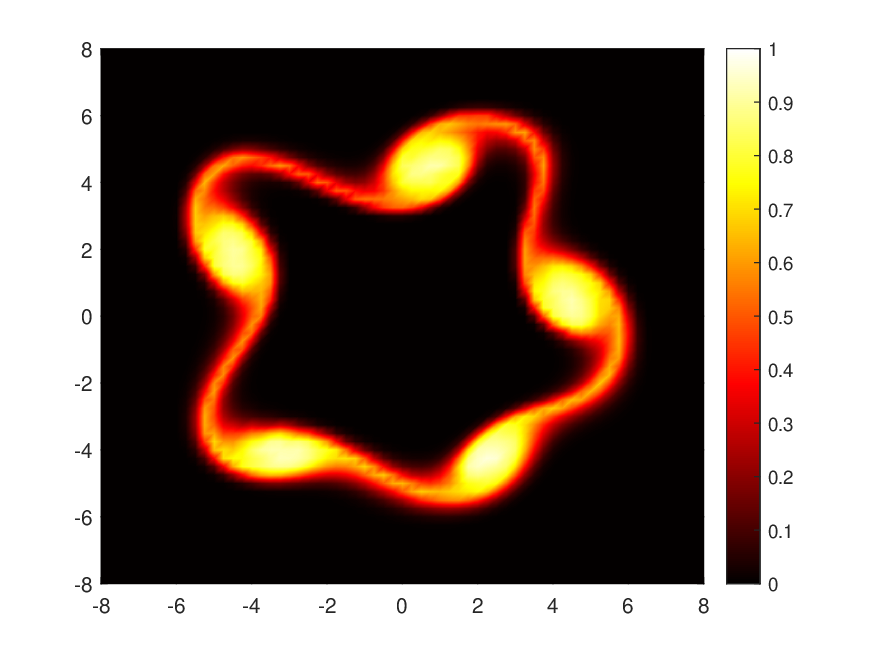}
}
\quad
\subfigure{
\includegraphics[scale=.4]{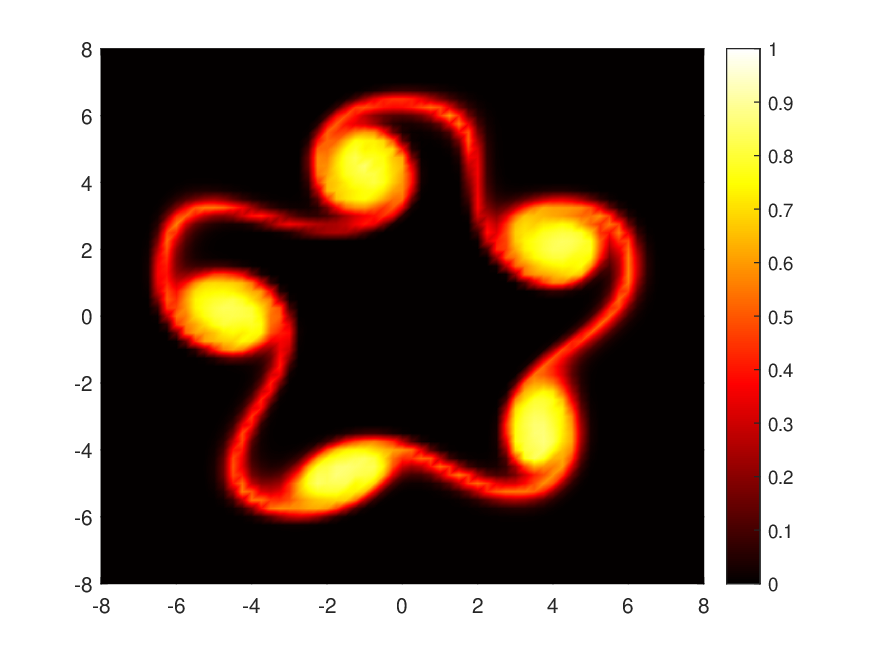}
}
\caption{Time evolution of the density $\rho$ for time $t=5,10,15,20$ and $l=5$ with parameter $\sigma=\tau=1,\varepsilon=10^{-4}$.}
\label{imgdi2}
\end{figure}

 \begin{figure}[!ht]
 \centering
 \subfigure{
 \includegraphics[scale=.4]{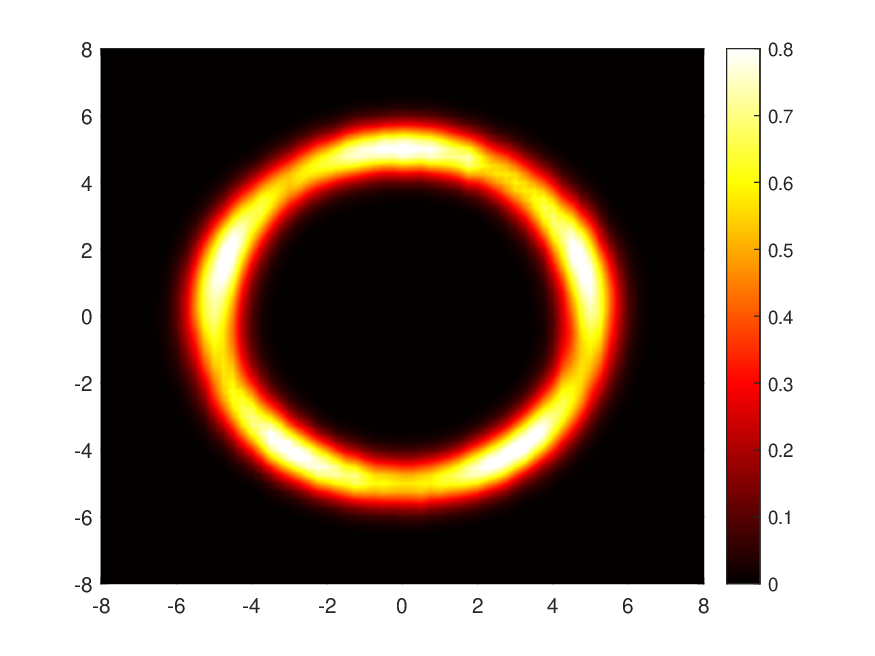}
 }
 \quad
 \subfigure{
 \includegraphics[scale=.4]{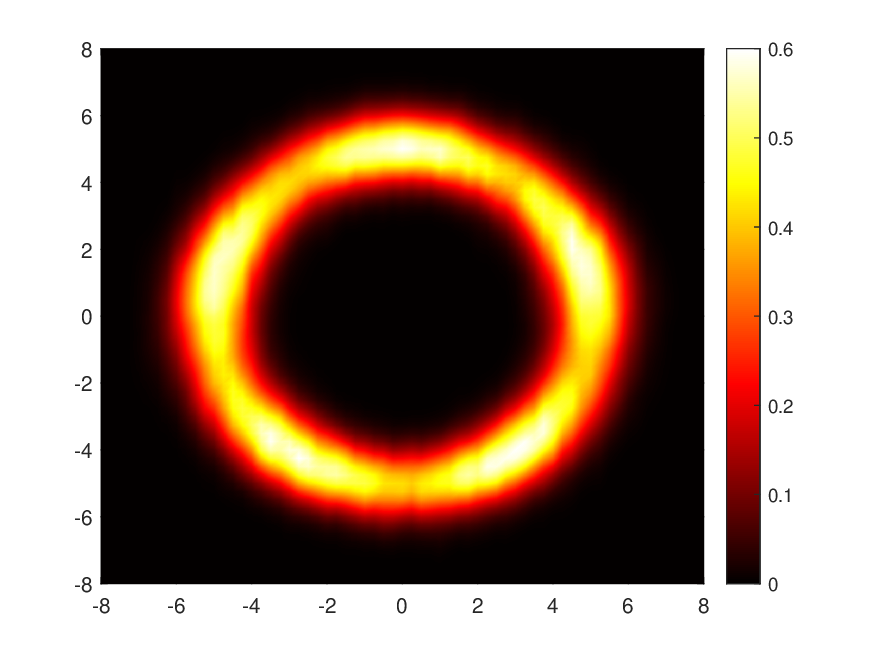}
 }
 \quad
 \subfigure{
 \includegraphics[scale=.4]{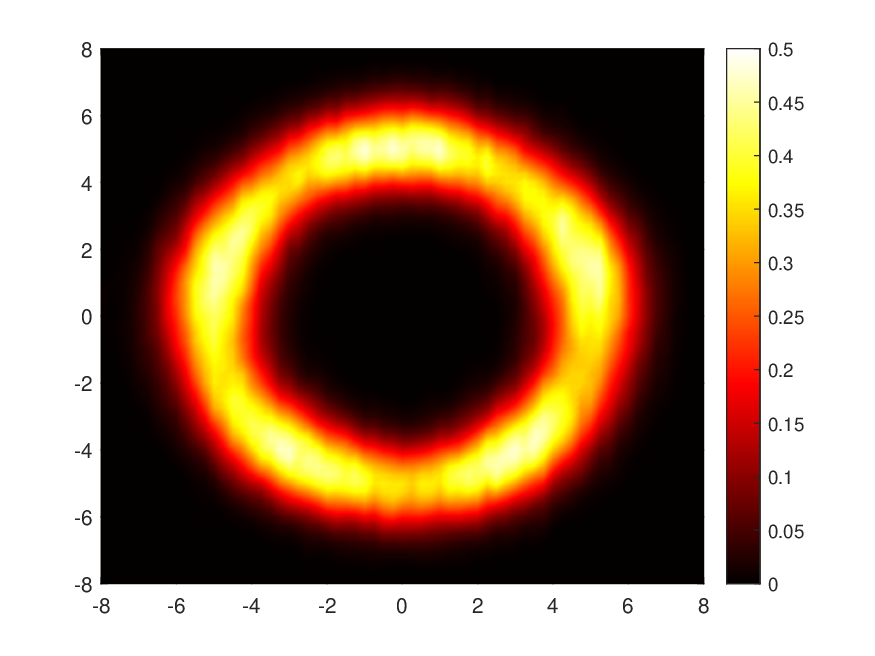}
 }
 \quad
 \subfigure{
 \includegraphics[scale=.4]{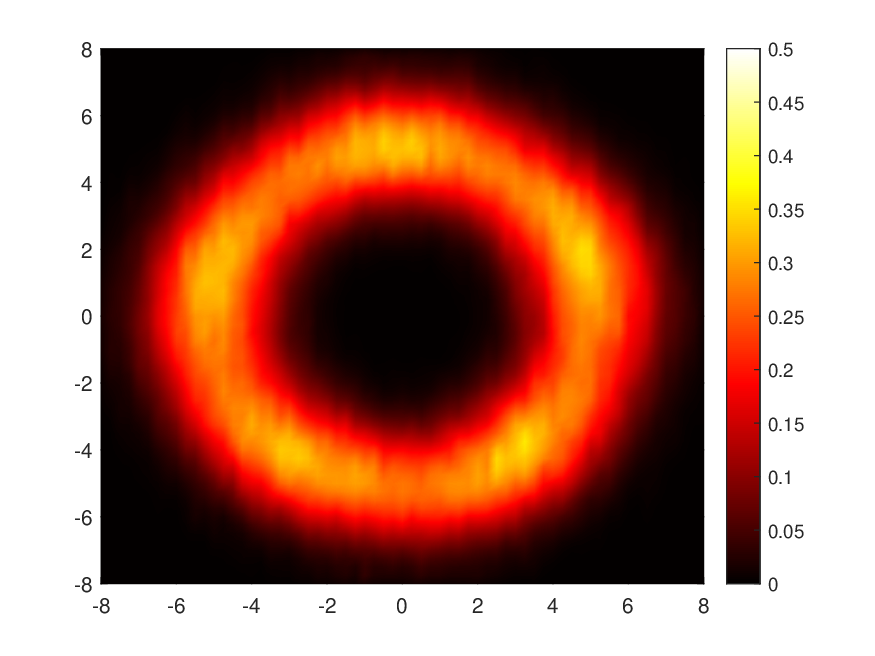}
 }
 \caption{Time evolution of the density $\rho$ for time $t=0.1,0.3,0.5,1$ and $l=5$ with parameter $\sigma=1,\tau=\varepsilon=10^{-2}$.}
 \label{imgdi3}
 \end{figure}

 \begin{figure}[!ht]
 \centering
 \subfigure{
 \includegraphics[scale=.4]{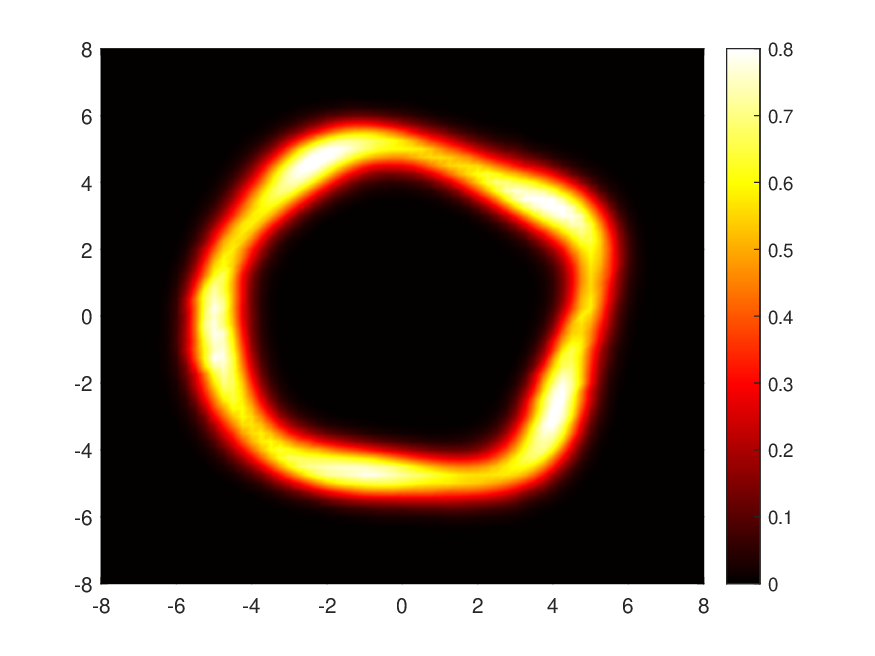}
 }
 \quad
 \subfigure{
 \includegraphics[scale=.4]{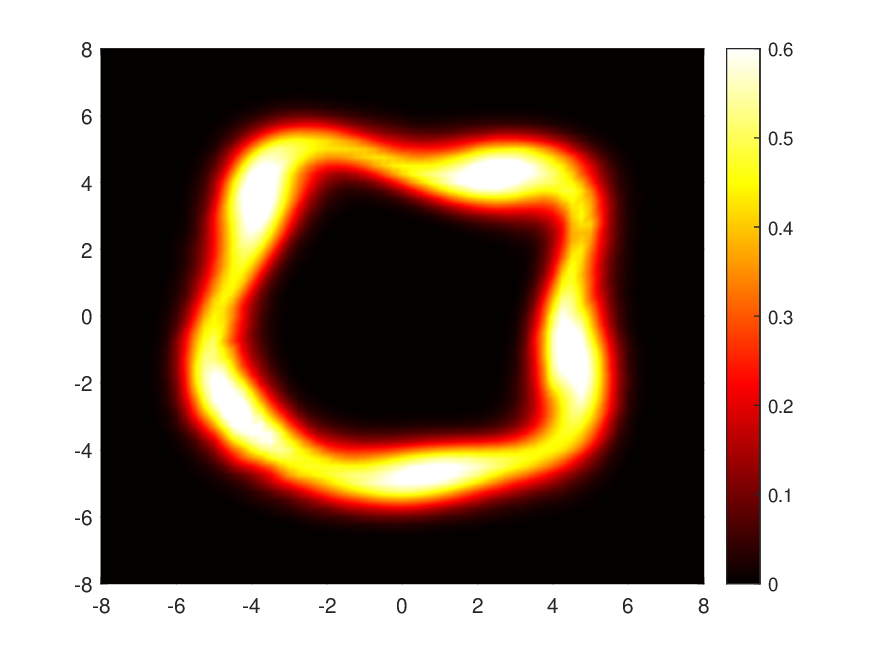}
 }
 \quad
 \subfigure{
 \includegraphics[scale=.4]{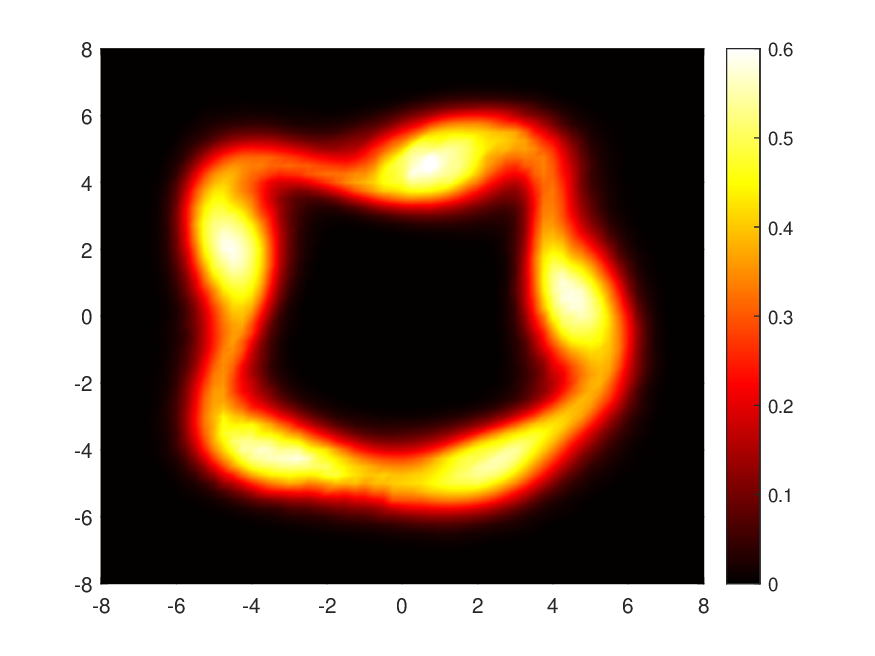}
 }
 \quad
 \subfigure{
 \includegraphics[scale=.4]{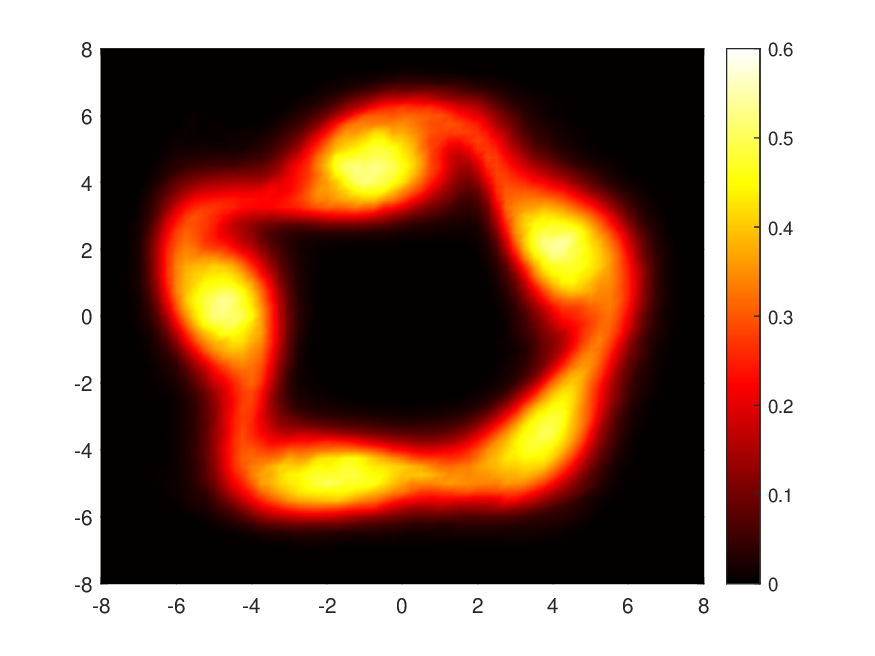}
 }
 \caption{Time evolution of the density $\rho$ for time $t=5,10,15,20$ and $l=5$ with parameter $\sigma=1,\tau=10^{-2},\varepsilon=10^{-4}$.}
 \label{imgdi4}
 \end{figure}

In Figure~\ref{imgdi1} and Figure~\ref{imgdi2}, we set a small oscillation with $\sigma=\tau=1$. 
It can be observed that collisions cause plasma to diffuse outward, while strong magnetic fields can limit this process.
By further enhancing collision in Figure~\ref{imgdi3} and Figure~\ref{imgdi4} with $\sigma=1,\tau=10^{-2}$, we can observe that the magnetic field with $\varepsilon=10^{-2}$ can no longer effectively suppress collision behavior.
Then it needs a stronger magnetic field to restrict the plasma.

\section{Conclusion}

In this study, we have developed novel PIC methods for solving the MVPFP system. Our paper begins with a thorough analysis of the asymptotic properties of the MVPFP equation, followed by the demonstration that the particle method effectively preserves these asymptotic properties across micro-to-macro scale transitions. 
The numerical framework combines finite element discretization in space with semi-implicit time integration, ensuring asymptotic-preserving characteristics that have been rigorously verified through both theoretical analysis and numerical experiments. 
This validation guarantees the long-term stability and accuracy of our proposed algorithms.
To illustrate the application of our method, we present comprehensive results from a $2+2$-dimensional simulation, particularly examining the influence of collision effects on long-term system behavior. 
Our numerical experiments demonstrate that the proposed methods accurately capture essential physical phenomena, including the magnetic confinement effects on plasma dynamics.

While the current work focuses on the maximal ordering scaling scenario, several important extensions remain for future investigation. 
These include the development of methods for more complex electromagnetic field configurations and the exploration of additional physical regimes. 
These challenging directions will form the basis of our ongoing research efforts.



\bibliography{references}
\bibliographystyle{plain}
\end{document}